\documentclass[11pt,reqno]{amsart}
\usepackage{amsmath}
\usepackage{cases}
\usepackage{mathrsfs}
\usepackage{bbm}
\usepackage{amssymb}
\usepackage{amscd}
\usepackage{amsfonts,latexsym,amsmath,
amsthm,amsxtra,mathdots,amssymb,latexsym,mathabx}
\usepackage[all,cmtip]{xy}
\RequirePackage{amsmath} \RequirePackage{amssymb}
\usepackage{color}
\usepackage{colordvi}
\usepackage{multicol}
\usepackage{hyperref}
\usepackage{mathtools}
\usepackage[margin=1in]{geometry}
\usepackage{xcolor}

\hypersetup{
    colorlinks,
    linkcolor={red!50!black},
    citecolor={blue!100!black},
    urlcolor={black}
}
\usepackage{cite}

\newcommand{\bea}{\begin{eqnarray}}
\newcommand{\eea}{\end{eqnarray}}
\newcommand{\bna}{\begin{eqnarray*}}
\newcommand{\ena}{\end{eqnarray*}}

\numberwithin{equation}{section}

\theoremstyle{plain}
\newtheorem{lemma}{Lemma}[section]
\newtheorem{theorem}[lemma]{Theorem}
\newtheorem{corollary}[lemma]{Corollary}

\theoremstyle{definition}

\newtheorem{remark}{Remark}

\newcommand{\cB}{\mathcal B}
\newcommand{\cC}{\mathcal C}

\newcommand{\cE}{\mathcal E}
\newcommand{\cS}{\mathcal S}

\newcommand{\Ad}{\mathrm{Ad}}
\newcommand{\new}{\mathrm{new}}
\newcommand{\old}{\mathrm{old}}

\newcommand{\cusp}{\mathrm{cusp}}

\renewcommand{\Re}{\operatorname{Re}}
\renewcommand{\Im}{\operatorname{Im}}

\newcommand{\GL}{\operatorname{GL}}

\newcommand{\eps}{\varepsilon}

\renewcommand{\bmod}[1]{\,(\mathrm{mod\,}{#1})}
\renewcommand{\i}{\mathrm{i}}
\makeatletter
\@namedef{subjclassname@2020}{%
  \textup{2020} Mathematics Subject Classification}

\title[]
{A $\mathrm{GL}_3$ level analog of Selberg's result for $S(t)$}

\author{Qingfeng Sun}
	\address{School of Mathematics and Statistics, Shandong University, Weihai\\Weihai, 264209, China}
	\email{qfsun@sdu.edu.cn}

\author{Hui Wang}
	\address{The University of Hong Kong Shenzhen Institute of Research and Innovation, Shenzhen, Guangdong 518057, China}
\email{wh0315@mail.sdu.edu.cn}



\thanks{Q. Sun was partially supported by the National Natural Science Foundation of China
(Grant No. 12471005) and
the Natural Science Foundation of Shandong Province (Grant No. ZR2023MA003).}

\date{}

\begin{document}
\begin{abstract}
Let $\pi$ be a spherical cuspidal $\mathrm{GL}_3$ newform of prime
level $N$ and trivial central character, and define
$S(t,\pi)=\pi^{-1}\arg L(1/2+\i t,\pi)$.
Assuming the generalized Riemann hypothesis for $L(s,\pi)$,
we establish an asymptotic formula for each fixed weighted moment
of $S(t,\pi)$ as $N\to\infty$, for fixed $t>0$.
Consequently, under the same weights,
$S(t,\pi)/\sqrt{\log\log N}$ converges in distribution to a centered
Gaussian distribution of variance $1/(2\pi^2)$.
The key input is a weighted orthogonality formula derived from
the Kuznetsov formula of prime level.
\end{abstract}
\maketitle
\enlargethispage{3pt}

\section{Introduction}
\setcounter{equation}{0}
For $T>0$, let $N_\zeta(T)$ count the nontrivial zeros of the
Riemann zeta function $\zeta(s)$ whose imaginary parts lie in $(0,T]$, counted with multiplicity.
When $T$ is not the ordinate of any zero, the Riemann--von Mangoldt formula reads
(see Davenport \cite[Chapter~15]{Davenport})
\bea\label{N(T)}
 N_\zeta(T)=\frac{T}{2\pi}\log\frac{T}{2\pi}
 -\frac{T}{2\pi}+\frac78+S(T)+O(T^{-1}),
\eea
where
$$S(T)=\frac1\pi\arg\zeta\!\left(\frac12+\i T\right).$$ 
The argument is continued from
$+\infty+\i T$, where its limiting value is $0$, to $1/2+\i T$.
At a zero ordinate, the value of $S(T)$ is the average of its
two one-sided limits, i.e.,
\bna
S(T)=\lim_{\varepsilon\rightarrow 0}\frac{S(T+\varepsilon)+S(T-\varepsilon)}{2}.
\ena 
For $T$ away from zero ordinates, \eqref{N(T)} identifies
$S(T)$, up to $O(T^{-1})$, with the deviation of the zero count
from its smooth main term. The size and oscillation of this
remainder have been studied extensively; see Backlund
\cite{BacklundS}, Titchmarsh \cite{TitchmarshS},
Carneiro--Chandee--Milinovich \cite{CCM}, and
Bondarenko--Seip \cite{BondarenkoSeip}.
The classical estimates
$S(T)\ll\log T$ and
$\int_0^T S(t)\,\mathrm{d}t\ll\log T$
(see \cite[Chapter~9]{Titchmarsh}) give a pointwise bound
and show that its mean on $[0,T]$ tends to zero.
To understand the typical size and distribution of these
fluctuations, Selberg \cite{Selberg,Selberg1} studied the
moments of $S(T)$; see also Goldston \cite{GoldstonS}.

Selberg \cite{Selberg,Selberg1} proved that, for each fixed
positive integer $m$,
\[
 \frac1T\int_T^{2T}S(t)^{2m}\,\mathrm{d}t
 =\frac{(2m)!}{m!(2\pi)^{2m}}(\log\log T)^m
 +O_m\bigl((\log\log T)^{m-1/2}\bigr).
\]
The constants appearing on the right-hand side coincide with the even moments of 
a Gaussian distribution of variance $1/(2\pi^2)$.  Selberg further investigated 
the quantity $S(t,\chi)=\frac{1}{\pi}\arg L(1/2+\i t,\chi)$ as $\chi$ ranges over primitive Dirichlet characters.  For a prime modulus $q$
and fixed $t>0$, Selberg \cite{Selberg2} obtained
\[
 \frac1q\sum_{\substack{\chi\bmod q\\\chi\ {\rm primitive}}}
 S(t,\chi)^{2m}
 =\frac{(2m)!}{m!(2\pi)^{2m}}(\log\log q)^m
 +O_{t,m}\bigl((\log\log q)^{m-1/2}\bigr).
\]
This motivates the study of the argument evaluated at a fixed point on the critical line, with the underlying automorphic form allowed to vary.

For Maass cusp forms on $\mathrm{SL}_2(\mathbb Z)$,
Hejhal--Luo \cite{HL} established a spectral analogue under the
generalized Riemann hypothesis (GRH).  Later, Liu--Liu \cite{LL}
gave the $\mathrm{SL}_3(\mathbb Z)$ counterpart,
and Chen--Lau--Wang \cite{CLW} considered the 
$\mathrm{GL}_n$  ($n\geq 3$) family setting.
These results address averages over varying archimedean parameters.
In the present paper we study spherical $\mathrm{GL}_3$
cuspidal newforms as their prime level tends to infinity.
Let $N$ be an odd prime. We write
\[
 \Gamma_0^{(3)}(N)=\left\{\gamma=(\gamma_{ij})_{1\le i,j\le 3}\in\mathrm{SL}_3(\mathbb Z):
 \gamma_{31}\equiv\gamma_{32}\equiv0\bmod N\right\},
\]
whose index is
\[
P_N:=[\mathrm{SL}_3(\mathbb Z):\Gamma_0^{(3)}(N)]
=N^2+N+1.
\]
Let $\mathcal B(N)$ denote an orthogonal basis of spherical cuspidal
newforms of level $N$ with trivial central character, or equivalently
the corresponding cuspidal automorphic representations of
$\mathrm{PGL}_3(\mathbb A_{\mathbb Q})$.
The precise normalization is given in Section~\ref{L-function}.
For $\pi\in\mathcal B(N)$, its standard $L$-function is
\[
 L(s,\pi)=\sum_{n=1}^{\infty}\frac{A_\pi(1,n)}{n^s},
 \qquad \Re(s)>1.
\] 
Write $\widetilde\pi$ for the contragredient of $\pi$.
The adjoint $L$-function in our spectral weights is defined by
\[
 L(s,\pi,\Ad)=\zeta(s)^{-1}L(s,\pi\times\widetilde\pi),
 \qquad\Re(s)>1,
\]
and its value at $s=1$ satisfies
\[
 L(1,\pi,\Ad)=\operatorname*{Res}_{s=1}L(s,\pi\times\widetilde\pi).
\]

For real $t$, we define
\[
 S(t,\pi)=\frac1\pi\arg L\!\left(\frac12+\i t,\pi\right),
\]
where the argument is continued along the horizontal line from
$+\infty+\i t$, where its limiting value equals $0$.  
If $1/2+\i t$ is a zero of $L(s,\pi)$, we define
$S(t,\pi)$ as the average of its two one-sided limits in $t$.
Under GRH for the standard
$L$-functions in this family, 
we 
obtain the following asymptotic formula.
\begin{theorem}\label{main-theorem}
Let $t>0$ and $n\in\mathbb N$ be fixed. Choose $A_0$ sufficiently large,
and fix  $h\in\widetilde{\mathcal H}(A_0)$ (see Section \ref{subsec:test-functions}), independently of $N$,
nonnegative on the automorphic spectrum and with $H_h>0$ (defined in \eqref{H_h}).
Put $$W_N=\sum_{\pi\in\cB(N)}\frac{h(\mu_\pi)}{L(1,\pi,\Ad)}.$$
Assume GRH.
For sufficiently large primes $N$, we have
\begin{equation*}
\begin{split}
\frac{1}{W_N}\sum_{\pi\in\cB(N)}
\frac{h(\mu_\pi)S(t,\pi)^n}{L(1,\pi,\Ad)}=
C_n(\log\log N)^{n/2}+O_{t,n,h}\big((\log\log N)^{(n-1)/2}\big),
\end{split}
\end{equation*}
where $C_n$ is defined by
\bna
C_n=\begin{cases}
\frac{n!}{(n/2)!(2\pi)^{n}}, \,&\textit{if }\, n\,\, \text{is even},\\
0, \,& \textit{if }\, n \,\,\text{is odd}.
\end{cases}
\ena
\end{theorem}
This indicates that $|S(t,\pi)|$ has order of magnitude $\sqrt{\log\log N}$ 
on average as $N\to\infty$. 
For fixed $t>0$ and the same fixed test function $h$, we define,
for every Borel set $E\subseteq\mathbb R$,
\[
 \nu_{N,t,h}(E):=\frac1{W_N}\sum_{\pi\in\cB(N)}
 \frac{h(\mu_\pi)}{L(1,\pi,\Ad)}
 \mathbf 1_E\!\left(\frac{S(t,\pi)}{\sqrt{\log\log N}}\right),
\]
where $\mathbf 1_E$ denotes
the indicator of $E$.  

\begin{corollary}\label{cor:gaussian-limit}
Under the assumptions of Theorem~\ref{main-theorem}, with
$A_0$ and $h$ fixed as above, the measures $\nu_{N,t,h}$ converge
weakly, as $N$ tends to infinity through primes, to
\[
 \mathrm{d}\nu(u)=\sqrt{\pi}\,e^{-\pi^2u^2}\,\mathrm{d}u.
\]
This is the Gaussian measure of mean zero and variance
$1/(2\pi^2)$.  In particular, for every real $a<b$,
\[
 \lim_{\substack{N\to\infty\\N\ {\rm prime}}}
\nu_{N,t,h}([a,b])
 =\sqrt{\pi}\int_a^b e^{-\pi^2u^2}\,\mathrm{d}u.
\]
\end{corollary}

We briefly contrast our result with unconditional work on other families:
 Liu--Shim \cite{LS} treated the holomorphic weight aspect;
Sun--Wang \cite{SWlevel, SWspectral, SWGL3} obtained analogous results for
holomorphic forms of weight $2$ in the prime level aspect, for level $1$ even $\mathrm{GL}_2$ Maass
cusp forms in the spectral aspect, and for level $1$
$\mathrm{GL}_3$ forms with generic spectral parameters, respectively; Sun--Wang--Yu \cite{SWY} considered
Dirichlet twists of a fixed holomorphic cusp form.
Note that those works rely on Selberg-type zero-density estimates tailored to their respective families. 
Our theorem concerns spherical $\mathrm{GL}_3$ cuspidal newforms of
prime level, averaged with a fixed spectral test function.
The GRH assumption supplies the control of zeros needed in the
approximation to $S(t,\pi)$.
A natural next step is therefore to establish an appropriate weighted 
zero-density estimate for this prime level $\mathrm{GL}_3$ family, 
with the goal of removing the GRH assumption from Theorem~\ref{main-theorem}, 
and we are currently pursuing this direction.
\\ \

\noindent{\bf Sketch of proof of Theorem \ref{main-theorem}.} 
The main input is the weighted orthogonality estimate for newforms
in Lemma~\ref{lem:new-orth}, derived from the Kuznetsov
formula of prime level in Blomer--Comtat \cite[Lemma~16]{BC}.
With the spectral test function fixed independently of $N$, we
estimate the geometric terms and separate the oldforms and the
continuous spectrum, uniformly in the level and the four Hecke
indices. The divisibility conditions on the Kloosterman moduli,
together with the corrected transform bounds in
\cite[Section~2.5]{BLv2}, make the Weyl terms negligible in the
required range. Coprimality of the indices to $N$ simplifies the
oldform coefficients, while the maximal Eisenstein spectrum
requires a uniform count of the inducing $\mathrm{GL}_2$ newforms
of level $N$. In the normalized formula, this latter contribution is
$O_{h,\varepsilon}(N^{-1+\varepsilon}Q^{7/64+\varepsilon})$,
where $Q$ is the product of the four Hecke indices, and determines
an admissible length for the prime sums.

The orthogonality estimate in Lemma~\ref{lem:new-orth} is
unconditional. The results in \cite{LL,CLW,SWGL3} are obtained
in the spectral aspect at fixed level and do not directly yield
the estimates uniform in the prime level $N$ required here.
These estimates are established in
Section~\ref{sec:proof-new-orth}, with a spectral test function
$h$ fixed independently of both $N$ and the moment order;
see Remark~\ref{rem:A0}.

Under GRH, the smoothed explicit formula in
Section~\ref{An approximation of S} allows us to approximate
$S(t,\pi)$ by the prime sum $\mathcal M(t,\pi)$ in \eqref{MR def},
whose length is chosen sufficiently short in terms of the fixed
moment order $n$. The Hecke relations and Lemma~\ref{lem:new-orth}
evaluate its weighted $n$-th moment. For even $n$, the main term
comes from distinct primes, each occurring once with its
coefficient and once with its conjugate, yielding the constant
$C_n$; all other patterns contribute to the error term.
We retain the full analytic conductor in the approximation and
control its contribution by Lemma~\ref{lem:conductor-error-moments}.
Together with the estimates for higher prime powers and smoothing
terms, this gives \eqref{eq:remainder-moment}: the normalized
weighted $2n$-th absolute moment of the remainder
$\mathcal R(t,\pi)$ is bounded uniformly in $N$.
H\"older's inequality then bounds the mixed terms by
$O_{h,t,n}((\log\log N)^{(n-1)/2})$, proving the moment asymptotic
for $S(t,\pi)$. Since the same test function $h$ is used for all
fixed moment orders and the Gaussian measure is determined by
its moments, the Gaussian limit follows by the method of moments.
\\ \

\noindent{\bf Structure of this paper.} This paper is organized as follows.
In Section~\ref{L-function}, we recall the Fourier coefficient normalizations
and the standard $L$-function of the prime level $\mathrm{GL}(3)$ forms in our family.
In Section~\ref{Kuznetsov formula}, we state the prime level Kuznetsov formula
and derive the weighted orthogonality estimate for short Hecke indices.
In Section~\ref{An approximation of S}, we establish the
smoothed approximation of $S(t,\pi)$ under the GRH.
In Section~\ref{proof of the main th}, we prove Theorem~\ref{main-theorem}.
The main contribution comes from the short prime sum $\mathcal M(t,\pi)$,
whose moments are evaluated using the orthogonality estimate and the Hecke relations.
We then bound the moments of the remainder $\mathcal R(t,\pi)$ and use
H\"older's inequality to absorb the mixed terms into the error term.
Finally, in Section~\ref{sec:proof-new-orth}, we prove the orthogonality
estimate in Lemma~\ref{lem:new-orth} from the Kuznetsov trace formula.
\\ \

\noindent{\bf Notation.} Throughout the paper,
$\varepsilon$ and $A$ denote arbitrarily small and arbitrarily large
positive real numbers, respectively, and may differ at each occurrence.
As usual, $e(z) = \exp (2 \pi \i z) = \textrm{e}^{2 \pi \i z}$.
We write $f = O(g)$ or $f \ll g$ to mean $|f| \leq cg$ for some unspecified positive constant $c$.
The symbol $f \asymp g$ means that $f=O(g)$ and $g=O(f)$.

\section{Standard $\rm GL(3)$ $L$-functions}\label{L-function}
Throughout the paper, $N$ denotes an odd prime.
For $\pi\in\mathcal B(N)$, write
\[
\pi=\sideset{}{'}\bigotimes_v \pi_v
\]
and denote its archimedean spectral parameter, in the convention of
Blomer--Comtat \cite{BC}, by
\[
\mu_\pi=(\mu_1,\mu_2,\mu_3),
\qquad
\mu_1+\mu_2+\mu_3=0.
\]
We assume throughout that $\pi_\infty$ is spherical.
We first fix the Fourier--Whittaker and Hecke normalizations at finite
primes, and then assemble the standard $L$-function.

\subsection{Fourier coefficients and Hecke normalization}

We normalize the Fourier--Whittaker coefficients of $\pi$ by
$A_\pi(1,1)=1$.  For every prime $q$, including the ramified prime $q=N$, express
$A_\pi(q^k,q^\ell)$ in terms of the Satake parameters
$\alpha_q,\beta_q,\gamma_q$ of $\pi$ as
\begin{equation}
A_\pi(q^k,q^\ell)
=\det
\begin{pmatrix}
\alpha_q^{k+\ell+2} & \beta_q^{k+\ell+2} & \gamma_q^{k+\ell+2}\\
\alpha_q^{k+1} & \beta_q^{k+1} & \gamma_q^{k+1}\\
1&1&1
\end{pmatrix}/
\det
\begin{pmatrix}
\alpha_q^2&\beta_q^2&\gamma_q^2\\
\alpha_q&\beta_q&\gamma_q\\
1&1&1
\end{pmatrix}.
\label{eq:local-Schur}
\end{equation}
This quotient is the Schur polynomial $s_{(k+\ell,k,0)}$ and is
interpreted by polynomial continuation when local parameters coincide.
The coefficients $A_\pi(m,n)$ are also multiplicative.

Suppose first that $q\nmid N$.  Then $\pi_q$ is unramified, and
\[
\alpha_q\beta_q\gamma_q=1,\qquad\quad
\{\alpha_q^{-1},\beta_q^{-1},\gamma_q^{-1}\}
=
\{\overline{\alpha_q},
  \overline{\beta_q},
  \overline{\gamma_q}\}.
\]
In addition, the Kim--Sarnak bound \cite{KS} toward the Ramanujan conjecture gives
\begin{equation}\label{KS bound}
|\alpha_q|,|\beta_q|,|\gamma_q|
\le q^{\theta_3},
\qquad
\theta_3=5/14.
\end{equation}
The Schur formula \eqref{eq:local-Schur} gives the first two Hecke
coefficients explicitly
\begin{align*}
A_\pi(1,q)
&=\alpha_q+\beta_q+\gamma_q,\\
A_\pi(q,1)
&=\alpha_q\beta_q+\alpha_q\gamma_q+\beta_q\gamma_q.
\end{align*}
Consequently, the local Euler factor satisfies
\begin{equation*}
L_q(s,\pi)^{-1}
=
1-A_\pi(1,q)q^{-s}
+A_\pi(q,1)q^{-2s}
-q^{-3s}.
\end{equation*}

For $(mn,N)=1$, unitarity also gives
\[
A_\pi(m,n)
=
\overline{A_\pi(n,m)}.
\]
Further, we have
\[
A_{\widetilde\pi}(m,n)=\overline{A_\pi(m,n)}
\quad\text{for all }m,n\geq1,
\qquad
A_{\widetilde\pi}(m,n)=A_\pi(n,m)
\quad\text{if }(mn,N)=1.
\]

At the ramified prime $N$, the coefficients are still given by
\eqref{eq:local-Schur}, but the local parameters take a different
form.  By Blomer--Comtat \cite[Lemma 1]{BC}, there is an imaginary number 
$\rho\in\i\mathbb R$ such that
\begin{equation*}
\alpha_N=N^{-1/2-\rho},
\qquad
\beta_N=N^{2\rho},
\qquad
\gamma_N=0.
\label{eq:ramified-Satake}
\end{equation*}
Substituting these parameters into \eqref{eq:local-Schur} gives
$A_\pi(N^j,1)=N^{-j(1/2-\rho)}$ and
$A_\pi(N^j,N^i)=A_\pi(N^j,1)A_\pi(1,N^i)$ for $i,j\geq0$.
Indeed, with $\gamma_N=0$ the Schur polynomial gives
$A_\pi(N^j,1)=(\alpha_N\beta_N)^j$, and
$\alpha_N\beta_N=N^{-1/2+\rho}$.
The bound in \eqref{KS bound} also holds at $N$ with $\theta_3=0$.
Since $\gamma_N=0$, the local Euler factor has two factors
\begin{equation}
L_N(s,\pi)^{-1}=(1-N^{-1/2-\rho-s})(1-N^{2\rho-s}).
\label{eq:ramified-Euler}
\end{equation}
\subsection{The standard $L$-function}

For $\Re(s)>1$, we define the standard $L$-function associated to $\pi$ by
\begin{equation*}\label{eq:standard L-function}
\begin{split}
L(s,\pi)=\sum_{n=1}^{\infty}\frac{A_\pi(1,n)}{n^s}.
\end{split}
\end{equation*}
It has the Euler product
\bna
L(s,\pi)=(1-N^{-1/2-\rho-s})^{-1}(1-N^{2\rho-s})^{-1}
\prod_{q\ne N}(1-A_\pi(1,q)q^{-s}
+A_\pi(q,1)q^{-2s}-q^{-3s})^{-1}.
\ena
Correspondingly,
the contragredient $\widetilde{\pi}$ has conjugate Dirichlet coefficients and local
parameters. For $\Re(s)>1$, we have
\begin{equation*}
\begin{split}
L(s,\widetilde\pi)&=\sum_{n=1}^{\infty}\frac{\overline{A_\pi(1,n)}}{n^s}\\
&=(1-N^{-1/2+\rho-s})^{-1}(1-N^{-2\rho-s})^{-1}\prod_{q\ne N}(1-\overline{A_\pi(1,q)}q^{-s}
+\overline{A_\pi(q,1)}q^{-2s}-q^{-3s})^{-1}.
  \end{split}
\end{equation*}
In particular, $\overline{A_\pi(1,n)}=A_\pi(n,1)$ when $(n,N)=1$;
at $N$, one must use the displayed dual factor instead.
Moreover, the completed $L$-function of $\pi$ is
\begin{equation}
\Lambda(s,\pi)
=N^{s/2}L_\infty(s,\pi)L(s,\pi),
\label{eq:completed-L}
\end{equation}
which admits an entire continuation to $s\in \mathbb{C}$ and satisfies the functional equation
\begin{equation}
\Lambda(s,\pi)
=\varepsilon(\pi)
\Lambda(1-s,\widetilde\pi).
\label{eq:functional-equation}
\end{equation}
Here $\varepsilon(\pi)$ is the root number of $L(s,\pi)$ with 
$|\varepsilon(\pi)|=1$, and the archimedean factor is 
\begin{equation*}
L_\infty(s,\pi)
=\pi^{-3s/2}\prod_{j=1}^{3}
\Gamma\!\left(\frac{s-\mu_j}{2}\right).
\label{eq:arch-factor}
\end{equation*}
The contragredient $\widetilde\pi$ has 
archimedean parameters $(-\mu_1,-\mu_2,-\mu_3)$, so
\[
L_\infty(s,\widetilde\pi)
=\pi^{-3s/2}\prod_{j=1}^{3}
\Gamma\!\left(\frac{s+\mu_j}{2}\right).
\]
For the general analytic properties of standard automorphic
$L$-functions, see Iwaniec--Kowalski \cite[Section~5.12]{IK}.

\section{The prime-level $\mathrm{GL}(3)$ Kuznetsov formula}
\label{Kuznetsov formula}

In this section we record the Kuznetsov trace formula for
$\Gamma_0^{(3)}(N)$, with $N$ prime.  We use the spectral parameter
$\mu_\pi$ and the kernel conventions of Blomer--Comtat \cite[Section~5]{BC}.
In particular, the same test function $h$ occurs on
the spectral side and in the integral transforms.
We use the probability measure on
\[
\Gamma_0^{(3)}(N)\backslash
\mathrm{SL}_3(\mathbb R)/\mathrm{SO}_3(\mathbb R).
\]
With this normalization the geometric side of the Kuznetsov formula
carries the factor $P_N$.

Let
\[
\mathfrak a_{\mathbb C}^*
=
\left\{
\mu=(\mu_1,\mu_2,\mu_3)\in\mathbb C^3:
\mu_1+\mu_2+\mu_3=0
\right\}.
\]
For $
\mu=(\mu_1,\mu_2,\mu_3)\in\mathfrak a_{\mathbb C}^*,$
we write
$
\int_{\Re(\mu)=0}
$
for integration over the tempered plane
\[
\Re(\mu_1)=\Re(\mu_2)=\Re(\mu_3)=0,
\qquad
\mu_1+\mu_2+\mu_3=0.
\]
\subsection{Test functions}\label{subsec:test-functions}
For $A_0>0$, let $\widetilde{\mathcal H}(A_0)$ denote the class of test
functions
$
h:\mathfrak a_{\mathbb C}^*\rightarrow\mathbb C
$
satisfying the following conditions:
\begin{enumerate}
\item $h$ is invariant under the Weyl group $\mathcal{W}\simeq S_3$;

\item $h$ is holomorphic for $|\Re(\mu_j)|\le A_0$, $j=1,2,3$;

\item for every $A>0$, $h(\mu)=O_{A,h}(\|\mu\|^{-A})$
as $\|\mu\|\to\infty$, uniformly in this strip;

\item $h$ vanishes whenever
\begin{equation*}
\mu_i-\mu_j=n\equiv1\bmod2,
\qquad
|n|\le A_0,
\qquad
1\le i<j\le3.
\end{equation*}
\end{enumerate}

This is the class in \cite[Definition~1]{BC}.  In the level-aspect
averages, $h$ is fixed independently of $N$, nonnegative on the
automorphic spectrum, and satisfies $H_h>0$.
The square of the function in \cite[Eq.\,(5.4)]{BC}
provides an example, since that function is real-valued on the
automorphic spectrum and is not identically zero on the tempered plane.

\subsection{Spectral density and Kuznetsov kernels}

Define the spectral density by
\begin{equation*}
\operatorname{spec}(\mu)
=\prod_{1\leq i<j\leq3}
(\mu_i-\mu_j)
\tan\!\left(\frac{\pi}{2}(\mu_i-\mu_j)\right).
\end{equation*}
In the diagonal spectral integral, we use
$\mathrm{d}\mu=\mathrm{d}\mu_1\,\mathrm{d}\mu_2$ with both imaginary
contours oriented upward.  Thus, if $\mu_j=\i t_j$ and
$t_3=-t_1-t_2$, then $\mathrm{d}\mu=-\mathrm{d}t_1\,\mathrm{d}t_2$;
more explicitly,
\[
\operatorname{spec}(\mu)\,\mathrm{d}\mu
=\prod_{1\leq i<j\leq3}
(t_i-t_j)\tanh\!\left(\frac{\pi}{2}(t_i-t_j)\right)
\,\mathrm{d}t_1\,\mathrm{d}t_2\geq0.
\]
Let $\epsilon_1,\epsilon_2\in\{\pm1\}.$
For $s\in\mathbb C$ and
$\mu\in\mathfrak a_{\mathbb C}^*$, put
\begin{equation*}
\begin{aligned}
\widetilde G^{\pm}(s,\mu)
:=\frac{\pi^{-3s}}{12288\pi^{7/2}}
\Bigg(
\prod_{j=1}^{3}
\frac{
\Gamma\left(\frac{s-\mu_j}{2}\right)}
{
\Gamma\left(\frac{1-s+\mu_j}{2}\right)}
\pm \i
\prod_{j=1}^{3}
\frac{
\Gamma\left(\frac{1+s-\mu_j}{2}\right)}
{
\Gamma\left(\frac{2-s+\mu_j}{2}\right)}
\Bigg).
\end{aligned}
\end{equation*}
For $s=(s_1,s_2)$ define
\begin{equation*}
\begin{aligned}
G_{\mathrm{sym}}^{\epsilon_1,\epsilon_2}(s,\mu)
:={}&
\frac{1}{1024\pi^{5/2}}
\sum_{\substack{
d_1,d_2,d_3\in\{0,1\}\\
d_3\equiv d_1+d_2\bmod 2
}}
\epsilon_1^{d_1}
\epsilon_2^{d_2}
(-1)^{d_1d_2}
\frac{
\Gamma\left(
\frac{1+d_3-s_1-s_2}{2}
\right)}
{
\Gamma\left(
\frac{d_3+s_1+s_2}{2}
\right)}
\\
&\times
\prod_{j=1}^{3}
\frac{
\Gamma\left(
\frac{d_1+s_1-\mu_j}{2}
\right)
\Gamma\left(
\frac{d_2+s_2+\mu_j}{2}
\right)}
{\Gamma\left(
\frac{1+d_1-s_1+\mu_j}{2}
\right)
\Gamma\left(
\frac{1+d_2-s_2-\mu_j}{2}
\right)}.
\end{aligned}
\end{equation*}
The Mellin--Barnes integrals below have the contour prescription
of Blomer--Leung \cite[Section~2.4]{BLv2}: contours are deformed to separate
the gamma poles and interpreted by analytic continuation where
necessary.  The notation $\int_{-\i\infty}^{\i\infty}$ does not
prescribe integrating through a gamma pole.
For $y\in\mathbb R^\times$, define the short-Weyl kernel
\begin{equation*}
K_{w_4}(y;\mu)
=\int_{-\i\infty}^{\i\infty}
|y|^{-s}
\widetilde G^{\operatorname{sgn}(y)}(s,\mu)
\frac{\mathrm{d}s}{2\pi \i}.
\end{equation*}
For $y=(y_1,y_2)\in(\mathbb R^\times)^2$, 
define the long-Weyl kernel
\begin{equation*}
\begin{aligned}
K_{w_6}^{\mathrm{sym}}(y;\mu)
={}&
\int_{-\i\infty}^{\i\infty}
\int_{-\i\infty}^{\i\infty}
(\pi^2|y_1|)^{-s_1}
(\pi^2|y_2|)^{-s_2}
G_{\mathrm{sym}}^{
\operatorname{sgn}(y_1),
\operatorname{sgn}(y_2)}
(s,\mu)
\frac{\mathrm{d}s_1\,\mathrm{d}s_2}{(2\pi \i)^2}.
\end{aligned}
\end{equation*}
For an admissible test function $h\in\widetilde{\mathcal H}(A_0)$,
define the integral transforms by
\begin{align*}
\Phi_{w_4}(y)
&=
\int_{\Re(\mu)=0}
h(\mu)
K_{w_4}(y;\mu)
\operatorname{spec}(\mu)\,\mathrm{d}\mu,
\\
\Phi_{w_5}(y)
&=
\int_{\Re(\mu)=0}
h(\mu)
K_{w_4}(-y;-\mu)
\operatorname{spec}(\mu)\,\mathrm{d}\mu,
\\
\Phi_{w_6}(y_1,y_2)
&=
\int_{\Re(\mu)=0}
h(\mu)
K_{w_6}^{\mathrm{sym}}
((y_1,y_2);\mu)
\operatorname{spec}(\mu)\,\mathrm{d}\mu.
\end{align*}

We record the following bounds for the integral transforms,
as given in the revised arXiv version of Blomer--Leung
\cite[Section~2.5, Lemmas~2.4, 2.5, and 2.6(a)]{BLv2}.
\begin{lemma}
Given $A,\varepsilon>0$ and
$j,j_1,j_2\in\mathbb N_0$, there is a choice of $A_0$ 
such that for $h\in \widetilde{\mathcal H}(A_0)$, we have
\begin{equation}
\label{eq:Phi-w4-bound}
|y|^j\Phi_{w_4}^{(j)}(y),
\quad
|y|^j\Phi_{w_5}^{(j)}(y)
\ll
|y|^{j/3}
\min\left(
|y|^A,
|y|^{-1/6+\varepsilon}
\right).
\end{equation}
Moreover, the derivatives satisfy
\[
\begin{aligned}
y_1^{j_1}y_2^{j_2}
 \frac{\partial^{j_1+j_2}}{\partial y_1^{j_1}\partial y_2^{j_2}}
 \Phi_{w_6}(y_1,y_2)
\ll
 (1+|y_1|^{1/2}+|y_1|^{1/3}|y_2|^{1/6})^{j_1}
 (1+|y_2|^{1/2}+|y_2|^{1/3}|y_1|^{1/6})^{j_2}.
\end{aligned}
\]
For the undifferentiated transform, with the prescribed $A,\varepsilon$, we have
\begin{equation}\label{eq:Phi-w6-small-bound}
 |\Phi_{w_6}(y_1,y_2)|
 \ll_{h,A,\varepsilon}
 \min\{|y_1^2y_2|,|y_1y_2^2|\}^{A}
 (1+|y_1|+|y_2|)^\varepsilon.
\end{equation}
Finally,
\begin{equation*}
\Phi_{w_6}(y_1,y_2)
\ll
\left(
2+|\log|y_1||+|\log|y_2||
\right)^3
(1+|y_1|)^{-1/4}
(1+|y_2|)^{-1/4}.
\end{equation*}
The implied constants may depend on
$h,A,j,j_1,j_2$ and $\varepsilon$.
\end{lemma}

\subsection{Kloosterman sums}

\leavevmode
For integers $n_1,n_2,m_1,m_2$ and positive integers $N,D_1,D_2$,
we define two types of Kloosterman sums. When $D_1\mid D_2$, put
\begin{equation*}
\label{eq:tilde-S-def}
\begin{aligned}
\widetilde S(n_1,n_2,m_1;D_1,D_2)
:=
\sum_{\substack{
C_1\;(\mathrm{mod}\ D_1),\
C_2\;(\mathrm{mod}\ D_2)\\
(C_1,D_1)=
(C_2,D_2/D_1)=1
}}
e\left(
n_2\frac{\overline C_1C_2}{D_1}
+
m_1\frac{\overline C_2}{D_2/D_1}
+
n_1\frac{C_1}{D_1}
\right).
\end{aligned}
\end{equation*}
Here $\overline C_1$ denotes the inverse of $C_1$ modulo $D_1$,
whereas $\overline C_2$ denotes the inverse of $C_2$ modulo
$D_2/D_1$.
For $N\mid D_1$ and $N\mid D_2$, we also define
\begin{equation*}
\begin{aligned}
&S^{(N)}(n_1,n_2,m_1,m_2;D_1,D_2)
\\
&:=
\sum_{\substack{
B_1,C_1\;(\mathrm{mod}\ D_1)\\
B_2,C_2\;(\mathrm{mod}\ D_2)\\
D_1C_2+B_1B_2+D_2C_1
\equiv0\;(\mathrm{mod}\ D_1D_2)\\
(B_j,C_j,D_j)=1\ (j=1,2)\\
N\mid B_1
}}
e\left(
\frac{
n_1B_1+
m_1(Y_1D_2-Z_1B_2)
}{D_1}
+
\frac{
n_2B_2+
m_2(Y_2D_1-Z_2B_1)
}{D_2}
\right),
\end{aligned}
\end{equation*}
where $Y_jB_j+Z_jC_j\equiv1\pmod{D_j}$ for $j=1,2$.
For later use, we record two bounds for nonzero integer arguments.
The first, due to Larsen \cite[Appendix, Theorem~1]{BFG}, is
\begin{equation*}
\widetilde S(n_1,n_2,m_1;D_1,D_2)
\ll_\varepsilon
\left(
(m_1,D_2/D_1)D_1^2,\,
(n_1,n_2,D_1)D_2
\right)
(D_1D_2)^\varepsilon.
\end{equation*}
The second, in the form stated by Blomer--Buttcane--Maga
\cite[Lemma~5, Eq.\,(4.1)]{BBM}, is
\begin{equation*}
S^{(N)}(n_1,n_2,m_1,m_2;D_1,D_2)
\ll_\varepsilon
(D_1D_2)^{1/2+\varepsilon}
\bigl(
(D_1,D_2)
(n_1m_1,[D_1,D_2])
(n_2m_2,[D_1,D_2])
\bigr)^{1/2}.
\end{equation*}
Here $[D_1,D_2]$ denotes the least common multiple.

\subsection{The spectral decomposition}


We decompose the spectral side into cuspidal and Eisenstein terms.

\subsubsection{Cuspidal new spectrum}

For $m_1,m_2,n_1,n_2\in\mathbb N$, let $\mathcal B(N)$ be the
basis of spherical cuspidal newforms introduced above, with the
Fourier coefficient conventions of \cite[Section~5.3]{BC}, and set
\begin{equation*}
\label{eq:Cnew}
\mathcal C_{\mathrm{new}}
:=
c_{\mathrm{cusp}}
\sum_{\pi\in\mathcal B(N)}
\frac{h(\mu_\pi)}
{L(1,\pi,\mathrm{Ad})}
\overline{A_\pi(m_1,m_2)}
A_\pi(n_1,n_2),
\end{equation*}
where $c_{\mathrm{cusp}}>0$ is an absolute normalization constant.

\subsubsection{Cuspidal old spectrum}

\leavevmode
For $\pi\in\mathcal B(1)$, let $B_\pi^{(j)}(m,n)$, $j=0,1,2$,
be the oldform coefficients defined in \cite[Eq.~(1.8)]{BC}.
The common $L^2$-normalization factor is separated from these
coefficients and included in the spectral weight below, as in
\cite[Corollary~4 and Section~5.3]{BC}. We put
\begin{equation*}
\mathcal C_{\mathrm{old}}
:=
c_{\mathrm{cusp}}
\sum_{\pi\in\mathcal B(1)}
\sum_{j=0}^{2}
\frac{h(\mu_\pi)}
{L(1,\pi,\mathrm{Ad})}
\overline{B_\pi^{(j)}(m_1,m_2)}
B_\pi^{(j)}(n_1,n_2).
\end{equation*}

The oldform coefficient formulae and Gram--Schmidt bounds in
\cite[Proposition~5 and Corollary~4]{BC} will be used below.

\subsubsection{Minimal Eisenstein series}

Following \cite[Proposition~6]{BC}, parameterize the minimal
Eisenstein spectrum by $s=(s_1,s_2)\in(\i\mathbb R)^2$ and set
$s_3=-s_1-s_2$. Let $B_s^{(j)}(m,n)$, $j=0,1,2$, denote the
oldform coefficients obtained from the Gram--Schmidt construction
in that proposition, with the common normalizing factor included
in the spectral density below.
The minimal Eisenstein contribution in the normalization of
\cite[Section~5.3]{BC} is
\begin{equation}
\label{eq:Emin}
\begin{aligned}
\cE_{\min}
:=
c_{\min}
\sum_{j=0}^{2}
\iint_{\Re(s)=0}
\frac{h(s)}
{
\displaystyle
\prod_{1\le i<\ell\le3}
|\zeta(1+s_i-s_\ell)|^2
}
\overline{B_s^{(j)}(m_1,m_2)}
B_s^{(j)}(n_1,n_2)
\,\mathrm{d}s_1\,\mathrm{d}s_2,
\end{aligned}
\end{equation}
where $c_{\min}>0$ is an absolute constant.  In these Eisenstein
spectral integrals, $\mathrm{d}s_j$ and $\mathrm{d}s$ denote
Lebesgue measure in the imaginary coordinates, with fixed measure
constants absorbed into $c_{\min}$ and $c_{\max}$.

\subsubsection{Maximal Eisenstein series}

Let $\phi$ be a spherical cuspidal $\mathrm{GL}_2$ Maass newform
of trivial central character, and let $t_\phi$ be its Langlands
parameter, normalized so that its Laplace eigenvalue is
$\frac{1}{4}-t_\phi^2$.  Thus $t_\phi\in\i\mathbb R$ in the tempered case
and $t_\phi\in[-7/64,7/64]$ otherwise
\cite[Section~1.4]{BC}.  For $s\in\i\mathbb R$,
the induced $\mathrm{GL}_3$ parameter is
\begin{equation*}
\mu_{\phi,s}
=
(s+t_\phi,\ s-t_\phi,\ -2s).
\end{equation*}
The same adjoint convention, including the Euler factor at $N$,
is used for $L(s,\phi,\Ad)$.

For $\phi$ of level $N$, write $A_{(\phi,s)}(m,n)$ for the
Fourier--Whittaker coefficients of the associated Eisenstein series
\cite[Proposition~7(a)]{BC}. We put
\begin{equation}
\label{eq:Emax-new}
\begin{aligned}
\cE_{\max,\mathrm{new}}
:=
c_{\max}
\sum_{\substack{\phi\,\operatorname{level}\,N}}
\int_{\Re(s)=0}
\frac{
h(s+t_\phi,s-t_\phi,-2s)
}{
L(1,\phi,\mathrm{Ad})
|L(1+3s,\phi)|^2
}
\overline{A_{(\phi,s)}(m_1,m_2)}
A_{(\phi,s)}(n_1,n_2)
\,\mathrm{d}s.
\end{aligned}
\end{equation}

For $\phi$ of level $1$, let $B_{(\phi,s)}^{(j)}(m,n)$,
$j=0,1,2$, be the level $N$ oldform coefficients of
\cite[Proposition~7(b)]{BC}, with the common normalizing factor
included in the spectral density below. Their contribution is
\begin{equation}
\label{eq:Emax-old}
\begin{aligned}
\cE_{\max,\mathrm{old}}
:=
c_{\max}
\sum_{\substack{\phi\,\operatorname{level}\,1}}
\sum_{j=0}^{2}
\int_{\Re(s)=0}
\frac{
h(s+t_\phi,s-t_\phi,-2s)
}{
L(1,\phi,\mathrm{Ad})
|L(1+3s,\phi)|^2
}
\overline{B_{(\phi,s)}^{(j)}(m_1,m_2)}
B_{(\phi,s)}^{(j)}(n_1,n_2)
\,\mathrm{d}s,
\end{aligned}
\end{equation}
where $c_{\max}>0$ is an absolute normalization constant.
\subsection{The Kuznetsov trace formula}
We may now state the Kuznetsov formula for the group 
$\Gamma_0^{(3)}(N)\subseteq\mathrm{SL}_3(\mathbb Z)$ for $N$ prime.

\begin{lemma}
\label{thm:prime-level-kuznetsov}
Let $N$ be an odd prime and let $
m_1,m_2,n_1,n_2\in\mathbb N.$
Let $
h\in\widetilde{\mathcal H}(10).$
There exist absolute constants
$
c_{\mathrm{cusp}},
c_{\min},
c_{\max}>0
$
such that
\begin{equation}
\label{eq:full-kuznetsov}
\frac{P_N}{N^2}\bigl(\mathcal C_{\mathrm{new}}+\cE_{\max,\mathrm{new}}\bigr)
+\mathcal C_{\mathrm{old}}+\cE_{\min}+\cE_{\max,\mathrm{old}}
=\Delta+\mathcal{S}_4+\mathcal{S}_5+\mathcal{S}_6,
\end{equation}
where
\begin{equation}\label{gs}
\begin{split}
\Delta
&=
P_N\,
\delta_{m_1,n_1}\delta_{m_2,n_2}
\frac{1}{192\pi^5}\int_{\Re(\mu)=0}
h(\mu)\operatorname{spec}(\mu)\,\mathrm{d}\mu.\\
\mathcal{S}_4
&=
P_N
\sum_{\epsilon=\pm1}
\sum_{\substack{
ND_2\mid D_1\\
m_2D_1=n_1D_2^2
}}
\frac{
\widetilde S(
-\epsilon n_2,m_2,m_1;
D_2,D_1)
}{
D_1D_2
}
\Phi_{w_4}
\left(
\frac{
\epsilon m_1m_2n_2
}{
D_1D_2
}
\right).\\
\mathcal{S}_5
&=
P_N
\sum_{\epsilon=\pm1}
\sum_{\substack{
N\mid D_1\mid D_2\\
m_1D_2=n_2D_1^2
}}
\frac{
\widetilde S(
\epsilon n_1,m_1,m_2;
D_1,D_2)
}{
D_1D_2
}
\Phi_{w_5}
\left(
\frac{
\epsilon n_1m_1m_2
}{
D_1D_2
}
\right).\\
\mathcal{S}_6
&=
P_N
\sum_{\epsilon_1,\epsilon_2=\pm1}
\sum_{\substack{
N\mid D_1\\
N\mid D_2
}}
\frac{
S^{(N)}
(
\epsilon_2n_2,\epsilon_1n_1,m_1,m_2;
D_1,D_2
)
}{
D_1D_2
}
\Phi_{w_6}
\left(
-\frac{
\epsilon_2m_1n_2D_2
}{
D_1^2
},
-\frac{
\epsilon_1m_2n_1D_1
}{
D_2^2
}
\right).
\end{split}
\end{equation}
\end{lemma}

\begin{remark}
With the displayed kernels, the diagonal constant is
$(192\pi^5)^{-1}$, as in \cite[Lemma~2.8]{BLv2}.
For $\pi\in\mathcal B(N)$, put
$D_\pi(v)=\sum_{m,n\geq1}|A_\pi(m,n)|^2(m^2n)^{-v}$.
The Schur formula and the local Weil--Deligne parameter give
\begin{align*}
\operatorname*{Res}_{v=1}D_\pi(v)
= &~\frac{N^2}{P_N\zeta(3)}L(1,\pi,\Ad).
\end{align*}
In Rankin--Selberg unfolding, the Eisenstein residue and quotient
volume contribute cancelling factors $P_N^{-1}$ and $P_N$.
For the fixed archimedean Whittaker function
and actual first Fourier coefficient $1$, the
squared probability norm
therefore acquires the factor $N^2/P_N$ relative to level $1$ (see
\cite[Theorems~3.5 and~3.7]{Jo} and
\cite[Section~4, proof of Theorem~4.1]{LM}).
Here $\zeta_N(u)=(1-N^{-u})^{-1}$, and the local
correction in this norm formula is
\begin{align*}
\frac{\zeta_N(3)L_N(1,\pi,\Ad)}
     {L_N(1,\pi\times\widetilde\pi)}
= &~\frac{\zeta_N(3)}{\zeta_N(1)}=\frac{N^2}{P_N}.
\end{align*}
For the maximal new spectrum, the compact induced
squared norm contributes
$N(N+1)/P_N$ by \cite[Lemma~10 and Eq.\,(3.6)]{BC}, while the
squared probability norm of the inducing
$\GL_2$ newform contributes $N/(N+1)$ under the same adjoint
convention (see \cite[Section~4]{LM} and \cite[Theorem~3.7]{Jo}).
Generic Fourier unfolding uses the unscaled Jacquet map,
with the local inducing vector normalized by
$W_{\phi,N}(1)=1$.
The induction isometry
\cite[Appendix~A, (A.1), (A.4), Proposition~A.2 and~(A.5)]{FLO}, together with
\cite[Theorem~3.7]{Jo}, gives the same squared local Jacquet factor
$|L_N(1+3s,\phi)|^{-2}$ as in the unramified case.
Here the canonical pairing in rank $r$ is
$\zeta_N(r)$ times Jo's mirabolic pairing, and
$(N(N+1)/P_N)\zeta_N(2)/\zeta_N(3)=1$.
Thus the remaining squared norm factor is $N^2/P_N$ in both new sectors,
and taking inverse squared norms gives $P_N/N^2$ in \eqref{eq:full-kuznetsov}.
\end{remark}

The trace formula yields the following weighted
orthogonality estimate for cuspidal newforms.  The proof is given
in Section~\ref{sec:proof-new-orth}.

\begin{lemma}
\label{lem:new-orth}
Fix $0\leq\eta<2$.  There exists $A_0=A_0(\eta)\geq10$
such that the following holds for every fixed
$h\in\widetilde{\mathcal H}(A_0)$. Let $N$ be an odd prime and let
$m_1,m_2,n_1,n_2\in\mathbb N$. Put $Q=m_1m_2n_1n_2$.
If $(Q,N)=1$ and $Q\leq N^\eta$, then for every $\eps>0$,
\begin{equation}
 c_{\cusp}\sum_{\pi\in\cB(N)}
 \frac{h(\mu_\pi)
 \overline{A_\pi(m_1,m_2)}A_\pi(n_1,n_2)}
 {L(1,\pi,\Ad)}=
 P_N H_h\,
 \delta_{m_1,n_1}\delta_{m_2,n_2}
 +O_{h,\eta,\eps}\!\left(
 N^{1+\eps}Q^{\theta_2+\eps}+Q^{2\theta_3+\varepsilon}
 \right),
 \label{eq:neworth}
\end{equation}
where $\theta_2=7/64$, $\theta_3=5/14$, and the normalized spectral integral is
\begin{equation}\label{H_h}
H_h=\frac{1}{192\pi^5}\int_{\Re(\mu)=0}h(\mu)\operatorname{spec}(\mu)\,\mathrm{d}\mu.
\end{equation}
\end{lemma}

\begin{remark}\label{rem:A0}
In the moment estimates in Section \ref{proof of the main th}, it suffices to apply
Lemma~\ref{lem:new-orth} with $\eta=1$.
For this choice, $A=3$ in
Sections~\ref{subsec:short-weyl}--\ref{subsec:long-weyl} gives
$\cS_4,\cS_5,\cS_6\ll_h N^{-2}$, while the remaining spectral
estimates use only the rapid decay of $h$ in the prescribed strip.
Thus $A_0$ and an admissible test function $h$ may be fixed
independently of the moment order. 
\end{remark}

\begin{corollary}\label{cor:mass}
Assume $H_h>0$.
Taking $m_1=m_2=n_1=n_2=1$ in Lemma~\ref{lem:new-orth}, we obtain
\[
 W_N=P_N H_h/c_{\cusp}+O_{h,\eps}(N^{1+\eps})
 =c_hN^2+O_{h,\eps}(N^{1+\eps}),
 \qquad c_h=H_h/c_{\cusp}>0.
\]
\end{corollary}
\section{An approximation of $S(t,\pi)$}
\label{An approximation of S}
In this section we will prove several lemmas and derive an approximation of $S(t,\pi)$.
We assume that the GRH holds for $L(s,\pi)$ and write $\rho=1/2+\i\gamma$
for a typical nontrivial zero of $L(s,\pi)$, equivalently
a zero of $\Lambda(s,\pi)$. All sums over $\rho$ count these zeros
with multiplicity.
Recall that the archimedean parameters satisfy
$|\Re(\mu_j)|\leq 5/14$ for $1\leq j\leq3$.
For $\Re(s)>1$, we have the Dirichlet series expansion for the logarithmic derivative of $L(s,\pi)$,
\bea\label{Log derivative definition}
-\frac{L'}{L}(s,\pi)=\sum_{n=1}^\infty \frac{\Lambda(n)C_\pi(n)}{n^s},
\eea
where $\Lambda(n)$ denotes the von Mangoldt function, and
\bea\label{Log derivative coefficient}
C_\pi(n)=\begin{cases}
\alpha_p^k+\beta_p^k+\gamma_p^k, 
\,&\text{if}\,\,n=p^k,\ k\geq1,\ p\not\mid N,\\
\alpha_N^k+\beta_N^k, 
\,&\text{if}\,\,n=N^k,\ k\geq1,\\
0, \,& \text{otherwise}.
\end{cases}
\eea
In particular,
$C_\pi(1)=0$. 
%
\begin{lemma}\label{lemma 1}
Let $x>1$. For $s\neq\rho$, $s\neq-2m+\mu_j$ ($m\geq0,\ 1\leq j\leq3$),
we have the following identity
\begin{equation*}
\begin{split}
\frac{L'}{L}(s,\pi)=&-\sum_{n\leq x^2} \frac{C_\pi(n)\Lambda_x(n)}{n^s}
+\frac{1}{\log x}\sum_\rho\frac{x^{\rho-s}(1-x^{\rho-s})}{(\rho-s)^2}\\
&+\frac{1}{\log x}\sum_{j=1}^3\sum_{m=0}^\infty\frac{x^{-2m+\mu_j-s}(1-x^{-2m+\mu_j-s})}
{(-2m+\mu_j-s)^2},
\end{split}
\end{equation*}
where
\bna
\Lambda_x(n)=\begin{cases}
\Lambda(n), \,&\textit{if }\,\,n\leq x,\\
\Lambda(n)\frac{\log(x^2/n)}{\log x}, \,& \textit{if }\,\,x\leq n\leq x^2,\\
0, \,& \textit{if }\,\,n\geq x^2.\\
\end{cases}
\ena
\end{lemma}
\begin{proof}
Recall the discontinuous integral
\bna
\frac{1}{2\pi \i}\int_{(\alpha)}\frac{y^s}{s^2}\mathrm{d}s=
\begin{cases}
\log y, \,&\textit{if }\,\, y\geq 1,\\
0, \,& \textit{if }\,\, 0<y\leq 1
\end{cases}
\ena
for $\alpha>0$.
Put $\alpha=\max\{2,2-\Re(s)\}$. 
Using \eqref{Log derivative definition}, we obtain
\[
(\log x)\sum_{n\leq x^2}\frac{C_\pi(n)\Lambda_x(n)}{n^s}=\frac{1}{2\pi\i}\int_{(\alpha)}
 \frac{x^v(1-x^v)}{v^2}\frac{L'}{L}(s+v,\pi)\,\mathrm{d}v.
\]
We move the contour all the way to the left, first letting the
heights of truncated rectangles tend to infinity along suitable
sequences avoiding the zeros, and then sending the left edge to
$-\infty$ while keeping it away from the trivial zeros.
The boundary integrals vanish by the standard bounds for $L'/L$,
the functional equation \eqref{eq:functional-equation}, and Stirling's formula, since $x>1$. The poles crossed are at $v=0$, $v=\rho-s$, and
$v=-2m+\mu_j-s$, counted with multiplicity. Thus
\[
\begin{aligned}
\frac{1}{2\pi\i}&\int_{(\alpha)}
 \frac{x^v(1-x^v)}{v^2}\frac{L'}{L}(s+v,\pi)\,\mathrm{d}v\\=&-\frac{L'}{L}(s,\pi)\log x
 +\sum_\rho\frac{x^{\rho-s}(1-x^{\rho-s})}{(\rho-s)^2}
 +\sum_{j=1}^3\sum_{m\geq0}
 \frac{x^{-2m+\mu_j-s}(1-x^{-2m+\mu_j-s})}
 {(-2m+\mu_j-s)^2}.
\end{aligned}
\]
Both sums converge absolutely under GRH. Comparing the two
expressions for the integral proves the identity.
\end{proof}
\begin{lemma}\label{lemma 2}
For $s=\sigma+\i t$ and $s'=\sigma'+\i t'$ such that $1/2\leq \sigma, \sigma'\leq 2$, $s\neq \rho$ and $s'\neq \rho$,
we have
\bna\label{Im part}
\Im\left(\frac{L'}{L}(s,\pi)-\frac{L'}{L}(s',\pi)\right)
=\Im\sum_\rho\left(\frac{1}{s-\rho}-\frac{1}{s'-\rho}\right)+O(1),
\ena
and, uniformly in the archimedean parameters,
\bea\label{Re part}
\Re\frac{L'}{L}(s,\pi)=\sum_\rho\frac{\sigma-1/2}{(\sigma-1/2)^2
+(t-\gamma)^2}+O(\log\mathcal Q_\pi(t)),
\eea
where
\[
 \mathcal Q_\pi(t):=N\prod_{j=1}^3(2+|t-\Im(\mu_j)|).
\]
\end{lemma}
\begin{proof}
By the Hadamard factorization theorem of the entire function $\Lambda(s,\pi)$ 
and \eqref{eq:completed-L}, we have
\[
\begin{aligned}
 \frac{L'}{L}(s,\pi)
 =b_\pi+\sum_\rho\left(\frac1{s-\rho}+\frac1\rho\right)
 -\frac12\sum_{j=1}^3\frac{\Gamma'}{\Gamma}
 \left(\frac{s-\mu_j}2\right)+\frac32\log\pi-\frac12\log N,
\end{aligned}
\]
for some $b_\pi$ with
$\Re(b_\pi)=-\sum_\rho\Re(1/\rho)$ (see \cite[Proposition~5.7(3)]{IK}).
For $z\notin\{0,-1,-2,\ldots\}$, we have
\bea\label{log derivative gamma}
\frac{\Gamma'}{\Gamma}(z)=-\gamma+\sum_{k\geq 1}\left(\frac{1}{k}-\frac{1}{k-1+z}\right),
\eea
where $\gamma$ is the Euler constant.
Using \eqref{log derivative gamma}, we get that, uniformly for $1/14\leq u\leq33/28$
and $v\in\mathbb R$,
\bna
\Re \frac{\Gamma'}{\Gamma}(u+\i v)=-\gamma+
\sum_{k\geq 1}\left(\frac{1}{k}-\frac{k+u-1}{(k+u-1)^2+v^2}\right)\ll \log(2+|v|),
\ena
and
\bna
\Im\frac{\Gamma'}{\Gamma}(u+\i v)=\sum_{k\geq 1}\frac{v}{(k+u-1)^2+v^2}\ll 1.
\ena
Taking the difference of the imaginary parts and the real part
of the Hadamard identity, respectively, and applying these
bounds proves the lemma.
\end{proof}
\begin{lemma}\label{lemma 3}
Let $x\geq4$ and put $\sigma_x=1/2+1/\log x$.
For $\sigma\geq\sigma_x$ and $t\in\mathbb R$, we have
\begin{equation*}
\begin{split}
\frac{L'}{L}(\sigma+\i t, \pi)=&-\sum_{n\leq x^2} \frac{C_\pi(n)\Lambda_x(n)}{n^{\sigma+\i t}}
+O\left(x^{1/2-\sigma}\left|\sum_{n\leq x^2} \frac{C_\pi(n)\Lambda_x(n)}{n^{\sigma_x+\i t}}\right|\right)\\
&+O\big(x^{1/2-\sigma}\log\mathcal Q_\pi(t)\big),
\end{split}
\end{equation*}
with absolute implied constants, uniformly in $\mu_\pi$.
\end{lemma}
\begin{proof}
Assume GRH, and recall that $|\Re(\mu_j)|\leq5/14$. 
For $\sigma\geq\sigma_x$, Lemma~\ref{lemma 1} gives
\begin{equation}\label{step 2}
\begin{aligned}
 \frac{L'}{L}(\sigma+\i t,\pi)
 ={}&-\sum_{n\leq x^2}\frac{C_\pi(n)\Lambda_x(n)}{n^{\sigma+\i t}}+2w_\pi(x,\sigma,t)x^{1/2-\sigma}
 \sum_\gamma\frac{\sigma_x-1/2}
 {(\sigma_x-1/2)^2+(t-\gamma)^2}\\
 &
 +O\left(\frac{x^{5/14-\sigma}}{\log x}\right),
\end{aligned}
\end{equation}
where $|w_\pi(x,\sigma,t)|\leq1$.
Taking real parts at $\sigma=\sigma_x$
in \eqref{step 2} and using \eqref{Re part}, we get
\begin{equation*}
\begin{split}
&\left(1-2\textrm{e}^{-1}\Re(w_\pi(x,\sigma_x,t))\right)
\sum_\gamma\frac{\sigma_x-1/2}
{(\sigma_x-1/2)^2+(t-\gamma)^2}\\
&\quad=-\Re\sum_{n\leq x^2} \frac{C_\pi(n)\Lambda_x(n)}{n^{\sigma_x+\i t}}
+O(\log\mathcal Q_\pi(t)).
\end{split}
\end{equation*}
Since $\left|1-2\textrm{e}^{-1}\Re(w_\pi(x,\sigma_x,t))\right|
\geq 1-2\textrm{e}^{-1}$,
we have 
\begin{equation}\label{step 4}
 \sum_\gamma\frac{\sigma_x-1/2}
 {(\sigma_x-1/2)^2+(t-\gamma)^2}
=O\left(\left|\sum_{n\leq x^2}
 \frac{C_\pi(n)\Lambda_x(n)}{n^{\sigma_x+\i t}}\right|\right)
 +O(\log\mathcal Q_\pi(t)).
\end{equation}
Substituting this into \eqref{step 2}, we complete the proof of the lemma.
\end{proof}
We now obtain the following approximation of $S(t,\pi)$.
\begin{theorem}\label{theorem 1}
Let $x\geq4$ and $\sigma_x=1/2+1/\log x$.  For $t\in\mathbb R$, we have
\[
\begin{aligned}
 S(t,\pi)=\frac1\pi\Im\sum_{2\leq n\leq x^2}
 \frac{C_\pi(n)\Lambda_x(n)}{n^{\sigma_x+\i t}\log n}+O\left(\frac1{\log x}
 \left|\sum_{n\leq x^2}\frac{C_\pi(n)\Lambda_x(n)}{n^{\sigma_x+\i t}}\right|\right)
 +O\left(\frac{\log\mathcal Q_\pi(t)}{\log x}\right),
\end{aligned}
\]
with absolute implied constants, uniformly in $\mu_\pi$.
\end{theorem}
\begin{proof}
Suppose first that $t$ is not a zero ordinate.  By definition,
\[
\begin{aligned}
 \pi S(t,\pi)
 =&-\int_{1/2}^\infty\Im \frac{L'}{L}(\sigma+\i t,\pi)\mathrm{d}\sigma\\
 ={}&-\int_{\sigma_x}^{\infty}\Im\frac{L'}{L}(\sigma+\i t,\pi)\,\mathrm{d}\sigma
 -(\sigma_x-1/2)\Im\frac{L'}{L}(\sigma_x+\i t,\pi)\\
 &+\int_{1/2}^{\sigma_x}\Im\left(
 \frac{L'}{L}(\sigma_x+\i t,\pi)-\frac{L'}{L}(\sigma+\i t,\pi)
 \right)\mathrm{d}\sigma\\
 =:&\,J_1+J_2+J_3.
\end{aligned}
\]
By Lemma \ref{lemma 3}, we have
\begin{equation*}
\begin{split}
J_1=&-\int_{\sigma_x}^\infty\Im \frac{L'}{L}(\sigma+\i t,\pi)\mathrm{d}\sigma\\
=&\int_{\sigma_x}^\infty\Im\sum_{n\leq x^2} 
\frac{C_\pi(n)\Lambda_x(n)}{n^{\sigma+\i t}}\mathrm{d}\sigma
+O\left(\left|\sum_{n\leq x^2} \frac{C_\pi(n)\Lambda_x(n)}{n^{\sigma_x+\i t}}\right|
\int_{\sigma_x}^\infty x^{1/2-\sigma}\mathrm{d}\sigma\right)\\
&+O\left(\log\mathcal Q_\pi(t)\int_{\sigma_x}^\infty x^{1/2-\sigma}\mathrm{d}\sigma\right)\\
=&\Im\sum_{2\leq n\leq x^2} \frac{C_\pi(n)\Lambda_x(n)}{n^{\sigma_x+\i t}\log n}
+O\left(\frac{1}{\log x}\left|\sum_{n\leq x^2} \frac{C_\pi(n)\Lambda_x(n)}{n^{\sigma_x+\i t}}\right|\right)
+O\left(\frac{\log\mathcal Q_\pi(t)}{\log x}\right).
\end{split}
\end{equation*}
Taking $\sigma=\sigma_x$ in Lemma \ref{lemma 3}, we have
\begin{equation*}
\begin{split}
|J_2|\leq&\,(\sigma_x-1/2)\left|\frac{L'}{L}(\sigma_x+\i t,\pi)\right|\\
=&\,O\left(\frac{1}{\log x}\left|\sum_{n\leq x^2} 
\frac{C_\pi(n)\Lambda_x(n)}{n^{\sigma_x+\i t}}\right|\right)
+O\left(\frac{\log\mathcal Q_\pi(t)}{\log x}\right).
\end{split}
\end{equation*}
Using Lemma \ref{lemma 2}, we get
\begin{equation*}
\begin{split}
&\Im\left( \frac{L'}{L}(\sigma_x+\i t,\pi)-\frac{L'}{L}(\sigma+\i t,\pi)\right)\\
=&\sum_\rho \frac{(\sigma_x-\sigma)(\sigma+\sigma_x-1)(t-\gamma)}
{\big((\sigma_x-1/2)^2+(t-\gamma)^2\big)\big((\sigma-1/2)^2+(t-\gamma)^2\big)}
+O(1).
\end{split}
\end{equation*}
Hence, by using \eqref{step 4}, we have
\begin{equation*}
\begin{split}
|J_3|\leq&\sum_\rho\int_{1/2}^{\sigma_x}
 \frac{(\sigma_x-1/2)^2|t-\gamma|}
{\big((\sigma_x-1/2)^2+(t-\gamma)^2\big)\big((\sigma-1/2)^2+(t-\gamma)^2\big)}\mathrm{d}\sigma
+O\bigg(\frac{1}{\log x}\bigg)\\
\leq&\,\frac{1}{\log x}\sum_\rho\frac{\sigma_x-1/2}{(\sigma_x-1/2)^2+(t-\gamma)^2}
\int_{1/2}^{\infty}\frac{|t-\gamma|}{(\sigma-1/2)^2+(t-\gamma)^2}\mathrm{d}\sigma
+O\bigg(\frac{1}{\log x}\bigg)\\
=&\,O\left(\frac{1}{\log x}\left|\sum_{n\leq x^2} \frac{C_\pi(n)\Lambda_x(n)}{n^{\sigma_x+\i t}}\right|\right)
+O\bigg(\frac{\log\mathcal Q_\pi(t)}{\log x}\bigg).
\end{split}
\end{equation*}
At a zero ordinate, one can take the average of the one-sided limits. Thus we complete the proof of the theorem.
\end{proof}
As a consequence of Theorem~\ref{theorem 1}, we obtain the
following bound for $S(t,\pi)$, uniformly in the archimedean parameters.
\begin{theorem}
For sufficiently large $N$, we have
\[
 S(t,\pi)\ll\frac{\log\mathcal Q_\pi(t)}{\log\log\mathcal Q_\pi(t)},
\]
uniformly for $\pi\in\mathcal B(N)$ and $t\in\mathbb R$.
\end{theorem}
\begin{proof}
The Satake bounds, including the ramified prime, give
$|C_\pi(p^k)|\leq3p^{5k/14}$.
Since $0\leq\Lambda_x(p^k)\leq\log p$, we have
\[
 |C_\pi(p^k)\Lambda_x(p^k)|\ll p^{5k/14}\log p.
\]
Thus,
\[
 \left|\Im\sum_{2\leq n\leq x^2}
 \frac{C_\pi(n)\Lambda_x(n)}{n^{\sigma_x+\i t}\log n}\right|
 \ll\sum_{p^k\leq x^2}\frac{p^{-k/7}}k
 \ll\frac{x^{12/7}}{\log x},
\]
and
\[
 \frac1{\log x}\left|\sum_{n\leq x^2}
 \frac{C_\pi(n)\Lambda_x(n)}{n^{\sigma_x+\i t}}\right|
 \ll\frac1{\log x}\sum_{p^k\leq x^2}p^{-k/7}\log p
 \ll\frac{x^{12/7}}{\log x}.
\]
The theorem follows by taking
$x=(\log\mathcal Q_\pi(t))^{7/12}$ in Theorem~\ref{theorem 1}.
\end{proof}
\section{Proofs of Theorem \ref{main-theorem} and Corollary~\ref{cor:gaussian-limit}}
\label{proof of the main th}
Throughout this section, we fix the test function $h$ as in Theorem~\ref{main-theorem} and write
\[
 \omega_{N,h}(\pi):=\frac{h(\mu_\pi)}{W_NL(1,\pi,\Ad)}.
\]
For sufficiently large $N$, these weights are nonnegative and
sum to $1$ by Corollary~\ref{cor:mass}.
For a parameter $x=N^{\delta/2}$ with 
a sufficiently small $0<\delta<1$ (such that $\log x\asymp\log N$), let
\bea\label{MR def}
\mathcal{M}(t,\pi):=\frac{1}{\pi}\Im\sum_{p\leq x^2}\frac{C_\pi(p)}{p^{1/2+\i t}} \qquad
\text{and}\qquad
\mathcal{R}(t,\pi):=S(t,\pi)-\mathcal{M}(t,\pi),
\eea
where, by \eqref{Log derivative coefficient},
\bna
C_\pi(p)=\begin{cases}
\alpha_p+\beta_p+\gamma_p, 
\,&\text{for a prime}\,\,p\not\mid N,\\
\alpha_N+\beta_N, 
\,&\text{for prime}\,\, p=N.
\end{cases}
\ena
For $p\leq x^2=N^\delta<N$ in $\mathcal M(t,\pi)$, 
we have $C_\pi(p)=A_\pi(1,p)$.

\subsection{The main term}\label{the main term}
Fix $n\in\mathbb N$ and choose $\delta=\delta(n)>0$ such that
\[
 0<\delta<\frac{32}{71n},\qquad x=N^{\delta/2}.
\]
With $A_0$ and $h$ fixed as in Remark~\ref{rem:A0}, we prove that
\begin{equation}\label{eq:main-moments}
 \sum_{\pi\in\cB(N)}\omega_{N,h}(\pi)\mathcal M(t,\pi)^n
 =C_n(\log\log N)^{n/2}
 +O_{h,n}\bigl((\log\log N)^{(n-1)/2}\bigr).
\end{equation}
Recall that
\[
 \mathcal M(t,\pi)^n
 =\frac{(-\i)^n}{(2\pi)^n}
 \left(\sum_{p\leq x^2}\frac{A_\pi(1,p)}{p^{1/2+\i t}}
 -\sum_{p\leq x^2}\frac{\overline{A_\pi(1,p)}}{p^{1/2-\i t}}\right)^n.
\]
On expanding the power and grouping by distinct primes, the terms,
apart from the common factor $(-\i)^n/(2\pi)^n$, have the form
\begin{equation}\label{eq:general-main-term}
 \frac{A_\pi(1,p_1)^{u(p_1)}\overline{A_\pi(1,p_1)}^{\,v(p_1)}}
 {p_1^{u(p_1)(1/2+\i t)}(-1)^{v(p_1)}p_1^{v(p_1)(1/2-\i t)}}
 \times\cdots\times
 \frac{A_\pi(1,p_\ell)^{u(p_\ell)}\overline{A_\pi(1,p_\ell)}^{\,v(p_\ell)}}
 {p_\ell^{u(p_\ell)(1/2+\i t)}(-1)^{v(p_\ell)}p_\ell^{v(p_\ell)(1/2-\i t)}},
\end{equation}
where $p_1<\cdots<p_\ell\leq x^2$, $u(p_j),v(p_j)\geq0$,
$u(p_j)+v(p_j)\geq1$, and $\sum_j(u(p_j)+v(p_j))=n$.
Here $p\leq x^2<N$, so every prime involved is unramified.
By the Hecke relations and \cite[Proposition~5.1]{LL}, we have
\begin{equation}\label{eq:Hecke-power-expansion}
 A_\pi(1,p)^u=\sum_{\substack{a,b\geq0\\a+b\leq u}}
 c_{a,b}(u)A_\pi(p^a,p^b),\qquad
 c_{a,b}(u)\in\mathbb Z_{\geq0},\qquad
 \sum_{a,b}c_{a,b}(u)\leq3^u.
\end{equation}
The coefficients $c_{a,b}(u)$ are independent of $p$, $N$ and $\pi$, and
\begin{equation}\label{eq:Hecke-power-congruence}
 c_{a,b}(u)\neq0\quad\Longleftrightarrow\quad(a,b)\in E_u,
\end{equation}
where
\[
 E_u=\big\{(n_2-n_1,n_3-n_2):
 n_3\geq n_2\geq n_1\geq0,\ n_j\in\mathbb Z,\ n_1+n_2+n_3=u\big\}.
\]
In particular, $(a,b)\in E_u$ implies $2a+b\equiv u\pmod3$.

Now we discuss the contribution from the general term \eqref{eq:general-main-term} 
in the following cases.

Case\,(\uppercase\expandafter{\romannumeral 1}):
\emph{Suppose $u(p_{j_0})\not\equiv v(p_{j_0})\pmod3$ for some $j_0$.}
By \eqref{eq:Hecke-power-expansion}, we have
\[
 A_\pi(1,p_{j_0})^{u(p_{j_0})}
 =\sum_{\substack{a,b\geq0\\a+b\leq u(p_{j_0})}}
 c_{a,b}(u(p_{j_0}))A_\pi(p_{j_0}^a,p_{j_0}^b)
\]
and
\[
 \overline{A_\pi(1,p_{j_0})}^{\,v(p_{j_0})}
 =\sum_{\substack{a,b\geq0\\a+b\leq v(p_{j_0})}}
 c_{a,b}(v(p_{j_0}))\overline{A_\pi(p_{j_0}^a,p_{j_0}^b)}.
\]
There is no pair $(a,b)$ for which both coefficients are nonzero.
Otherwise, \eqref{eq:Hecke-power-congruence} would give
$(a,b)\in E_{u(p_{j_0})}\cap E_{v(p_{j_0})}$, and hence
\[
 u(p_{j_0})\equiv2a+b\equiv v(p_{j_0})\pmod3,
\]
contrary to the assumption.
Expanding at every prime and using multiplicativity, the numerator
of \eqref{eq:general-main-term} is therefore a linear combination of
$\overline{A_\pi(m_1,m_2)}A_\pi(n_1,n_2)$ with
$(m_1,m_2)\neq(n_1,n_2)$ and total coefficient sum at most $3^n$.
Moreover, in this case,
\[
 Q=m_1m_2n_1n_2\leq\prod_{j=1}^{\ell}p_j^{u(p_j)+v(p_j)}
 \leq x^{2n}=N^{n\delta}<N.
\]
Thus $(Q,N)=1$, and Lemma~\ref{lem:new-orth} applies with $\eta=1$.
Since $\delta_{m_1,n_1}\delta_{m_2,n_2}=0$,
Lemma~\ref{lem:new-orth} and Corollary~\ref{cor:mass}, together with
$2\theta_3<1$, bound the contribution from the general term
\eqref{eq:general-main-term} to \eqref{eq:main-moments} by
\begin{equation}\label{eq:case-I-bound}
\begin{aligned}
 &\ll_{h,n,\varepsilon}
 \sum_{p_1<\cdots<p_\ell\leq x^2}
 p_1^{-(u(p_1)+v(p_1))/2}\cdots
 p_\ell^{-(u(p_\ell)+v(p_\ell))/2}
 \left(N^{-1+\varepsilon}x^{2n(\theta_2+\varepsilon)}
       +N^{-2}x^{2n}\right)\\
 &\ll_{h,n,\varepsilon}
 \left(\sum_{p\leq x^2}p^{-1/2}\right)^\ell
 \left(N^{-1+\varepsilon}x^{2n(\theta_2+\varepsilon)}
       +N^{-2}x^{2n}\right)\\
 &\ll_{h,n,\varepsilon}
 N^{-1+\varepsilon+n\delta(1/2+\theta_2+\varepsilon)}
 +N^{-2+3n\delta/2},
\end{aligned}
\end{equation}
which is negligible by our choice of $\delta$.
Here we used $\ell\leq n$ and $\sum_{p\leq x^2}p^{-1/2}\ll x$.

Case\,(\uppercase\expandafter{\romannumeral 2}):
\emph{Suppose $u(p_j)\equiv v(p_j)\pmod3$ for all $j$, and
$u(p_{j_0})\geq2$ or $v(p_{j_0})\geq2$ for some $j_0$.}
Then $u(p_j)+v(p_j)\geq2$ for every $j$, and
$u(p_{j_0})+v(p_{j_0})\geq3$, so
$\ell\leq\lfloor(n-1)/2\rfloor$.
Expanding the numerator of \eqref{eq:general-main-term} by
\eqref{eq:Hecke-power-expansion} and multiplicativity, terms
$\overline{A_\pi(m_1,m_2)}A_\pi(n_1,n_2)$ with
$(m_1,m_2)=(n_1,n_2)$ may now occur.
By Lemma~\ref{lem:new-orth} and Corollary~\ref{cor:mass},
their total diagonal contribution for each fixed tuple of primes
is $O_n(1)$, while the accumulated trace error, including the
normalization error, is bounded as in \eqref{eq:case-I-bound}.
Hence the contribution from the general term \eqref{eq:general-main-term} 
to \eqref{eq:main-moments} is
\[
\begin{aligned}
 &\ll_{h,n}\sum_{p_1<\cdots<p_\ell\leq x^2}
 p_1^{-(u(p_1)+v(p_1))/2}\cdots
 p_\ell^{-(u(p_\ell)+v(p_\ell))/2}\\
 &\ll_{h,n}\bigg(\sum_{p\leq x^2}
 p^{-1}\bigg)^\ell
 \ll_{h,n}(\log\log N)^{\lfloor(n-1)/2\rfloor},
\end{aligned}
\]
where we used the classical estimate
\bna
\sum_{p\leq x}\frac{1}{p}=\log\log x+O(1).
\ena
In particular, Cases~(\uppercase\expandafter{\romannumeral 1}) and~(\uppercase\expandafter{\romannumeral 2}) show that, for odd $n$,
\[
 \sum_{\pi\in\cB(N)}\omega_{N,h}(\pi)\mathcal M(t,\pi)^n
 \ll_{h,n}(\log\log N)^{(n-1)/2}.
\]

Case\,(\uppercase\expandafter{\romannumeral 3}):
\emph{Suppose $u(p_j)=v(p_j)=1$ for all $j$.}
Then $n=2m$ and $\ell=m$. The total contribution from these terms to \eqref{eq:main-moments} is
\[
 \binom{2m}{m}(m!)^2\frac{(-1)^m(-\i)^{2m}}{(2\pi)^{2m}}
 \sum_{p_1<\cdots<p_m\leq x^2}\frac1{p_1\cdots p_m}
 \sum_{\pi\in\cB(N)}\omega_{N,h}(\pi)
 \prod_{j=1}^m|A_\pi(1,p_j)|^2.
\]
By the Hecke relations,
\[
 \prod_{j=1}^m|A_\pi(1,p_j)|^2
 =|A_\pi(1,p_1\cdots p_m)|^2.
\]
Applying Lemma~\ref{lem:new-orth} and Corollary~\ref{cor:mass} with
$(m_1,m_2)=(n_1,n_2)=(1,p_1\cdots p_m)$, we have
\[
\begin{aligned}
 &\frac{(2m)!}{(2\pi)^{2m}}
 \sum_{p_1<\cdots<p_m\leq x^2}\frac1{p_1\cdots p_m}\times\left(1+O_{h,\varepsilon}\left(
 N^{-1+\varepsilon}(p_1\cdots p_m)^{2\theta_2+2\varepsilon}
 +N^{-2}(p_1\cdots p_m)^2\right)\right)\\
 &=\frac{(2m)!}{m!(2\pi)^{2m}}
 \sum_{\substack{p_1,\ldots,p_m\leq x^2\\p_j\ \mathrm{distinct}}}
 \frac1{p_1\cdots p_m}+O_{h,n,\varepsilon}\bigg(
 N^{-1+\varepsilon}
 \sum_{\substack{p_1,\ldots,p_m\leq x^2\\p_j\ \mathrm{distinct}}}
(p_1\cdots p_m)^{-1+2\theta_2+2\varepsilon}\bigg)\\
 &\quad+O_{h,n}\bigg(
 N^{-2}\sum_{\substack{p_1,\ldots,p_m\leq x^2\\p_j\ \mathrm{distinct}}}
 p_1\cdots p_m\bigg)\\
 &=\frac{(2m)!}{m!(2\pi)^{2m}}\bigl(\log\log N+O_n(1)\bigr)^m
 +O_{h,n,\varepsilon}\left(
 N^{-1+\varepsilon}x^{4m(\theta_2+\varepsilon)}+N^{-2}x^{4m}\right)\\
 &=C_{2m}(\log\log N)^m
 +O_{h,n}((\log\log N)^{m-1}).
\end{aligned}
\]

Now the asymptotic formula \eqref{eq:main-moments} follows from
Cases (\uppercase\expandafter{\romannumeral 1})--(\uppercase\expandafter{\romannumeral 3}).


\subsection{The remainder term}
We first record a bound for the error term in
Theorem~\ref{theorem 1}.
\begin{lemma}
\label{lem:conductor-error-moments}
For fixed $t$, $r>0$, and $0<\delta<1$, let $x=N^{\delta/2}$.
Then, as $N$ tends to infinity through primes,
\begin{equation*}\label{eq:conductor-error-moments}
 \sum_{\pi\in\cB(N)}\omega_{N,h}(\pi)
 \left(\frac{\log\mathcal Q_\pi(t)}{\log x}\right)^r
 \ll_{h,t,r,\delta}1.
\end{equation*}
\end{lemma}
\begin{proof}
With $\|\mu_\pi\|^2=\sum_{j=1}^3|\mu_j|^2$, we have
\[
 \mathcal Q_\pi(t)=N\prod_{j=1}^3(2+|t-\Im(\mu_j)|)
 \leq N(2+|t|)^3(1+\|\mu_\pi\|)^3.
\]
Since $x=N^{\delta/2}$ and $N\geq3$, it follows that
\[
\begin{aligned}
 \left(\frac{\log\mathcal Q_\pi(t)}{\log x}\right)^r
 &\ll_{t,r,\delta}\bigl(1+\log(1+\|\mu_\pi\|)\bigr)^r
 \ll_{t,r,\delta}1+\|\mu_\pi\|^2.
\end{aligned}
\]
By unitarity, $\{\mu_j\}=\{-\overline{\mu_j}\}$
as multisets, so $\sum_j\mu_j^2$ is real.
Using $|\Re(\mu_j)|\leq5/14$, we obtain
\[
2(1+\|\mu_\pi\|)^2\geq2-\sum_{j=1}^3\mu_j^2
 =2+\|\mu_\pi\|^2-2\sum_{j=1}^3(\Re(\mu_j))^2
 \geq1+\|\mu_\pi\|^2.
\]
Further, the function $\bigl(2-\sum_{j=1}^3\mu_j^2\bigr)h(\mu)$ belongs
to $\widetilde{\mathcal H}(A_0)$ and is fixed independently of $N$. Applying Lemma~\ref{lem:new-orth}
with this test function and $m_1=m_2=n_1=n_2=1$ gives
\[
 0\leq\sum_{\pi\in\cB(N)}
 \frac{\bigl(2-\sum_{j=1}^3\mu_j^2\bigr)h(\mu_\pi)}
 {L(1,\pi,\Ad)}\ll_h N^2.
\]
Finally, $W_N\asymp_h N^2$ by Corollary~\ref{cor:mass}.
This bound and the nonnegativity of the weights yield
\[
\begin{aligned}
 \sum_{\pi\in\cB(N)}\omega_{N,h}(\pi)
 \left(\frac{\log\mathcal Q_\pi(t)}{\log x}\right)^r\ll_{t,r,\delta}\frac1{W_N}
 \sum_{\pi\in\cB(N)}
 \frac{\bigl(2-\sum_{j=1}^3\mu_j^2\bigr)h(\mu_\pi)}
 {L(1,\pi,\Ad)}
 \ll_{h,t,r,\delta}1.
\end{aligned}
\]
\end{proof}

With $x=N^{\delta/2}$ and $\delta=\delta(n)$ fixed as in
Section~\ref{the main term}, we prove
\begin{equation}\label{eq:remainder-moment}
 \sum_{\pi\in\cB(N)}\omega_{N,h}(\pi)|\mathcal R(t,\pi)|^{2n}
 \ll_{h,t,n}1.
\end{equation}
By the definition \eqref{MR def} and Theorem \ref{theorem 1}, we have
\begin{equation}\label{R(t)}
\begin{aligned}
 \mathcal R(t,\pi)=&\,S(t,\pi)
 -\frac1\pi\Im\sum_{p\leq x^2}
 \frac{C_\pi(p)}{p^{1/2+\i t}}\\
 =&\,\frac1\pi\Im\sum_{p\leq x^2}
 \frac{C_\pi(p)
       \bigl(\Lambda_x(p)p^{1/2-\sigma_x}-\Lambda(p)\bigr)}
 {p^{1/2+\i t}\log p}+\frac1\pi\Im\sum_{k=2}^{\infty}\sum_{p^k\leq x^2}
 \frac{C_\pi(p^k)\Lambda_x(p^k)}
 {p^{k(\sigma_x+\i t)}\log p^k}\\
 &+O\left(\frac1{\log x}
 \left|\sum_{k=1}^{\infty}\sum_{p^k\leq x^2}
 \frac{C_\pi(p^k)\Lambda_x(p^k)}
 {p^{k(\sigma_x+\i t)}}\right|\right)
 +O\left(\frac{\log\mathcal Q_\pi(t)}{\log x}\right).
\end{aligned}
\end{equation}
Note that all primes with $p^k\leq x^2<N$ are unramified. 
By the bound
$|C_\pi(p^k)|\leq3p^{5k/14}$,
\[
 \sum_{k\geq8}\sum_{p^k\leq x^2}
 \frac{|C_\pi(p^k)|\Lambda_x(p^k)}{p^{k\sigma_x}}
 \leq3\sum_p(\log p)\sum_{k\geq8}p^{-k/7}
 \ll\sum_p\frac{\log p}{p^{8/7}}<\infty.
\]
Thus the terms with $k\geq8$ in \eqref{R(t)}
contribute $O(1)$, uniformly in $\pi$.
Separating the prime terms from the higher prime powers, using
$\sum_\pi\omega_{N,h}(\pi)=1$ and Lemma \ref{lem:conductor-error-moments}, we obtain
\begin{equation}\label{eq:remainder-decomposition}
\begin{aligned}
 \sum_{\pi\in\cB(N)}&\omega_{N,h}(\pi)|\mathcal R(t,\pi)|^{2n}\\\ll_{h,t,n}&
 \sum_{\pi\in\cB(N)}\omega_{N,h}(\pi)
 \left|\Im\sum_{p\leq x^2}
 \frac{C_\pi(p)
       \bigl(\Lambda_x(p)p^{1/2-\sigma_x}-\Lambda(p)\bigr)}
 {p^{1/2+\i t}\log p}\right|^{2n}\\
 +\,&\sum_{k=2}^7\sum_{\pi\in\cB(N)}\omega_{N,h}(\pi)
 \left|\Im\sum_{p^k\leq x^2}
 \frac{C_\pi(p^k)\Lambda_x(p^k)}
 {p^{k(\sigma_x+\i t)}\log p^k}\right|^{2n}\\
 +\,&\frac1{(\log x)^{2n}}
 \sum_{k=1}^7\sum_{\pi\in\cB(N)}\omega_{N,h}(\pi)
 \left|\sum_{p^k\leq x^2}
 \frac{C_\pi(p^k)\Lambda_x(p^k)}{p^{k(\sigma_x+\i t)}}\right|^{2n}+1.
\end{aligned}
\end{equation}
For the three sums in \eqref{eq:remainder-decomposition}, we have
\[
\begin{aligned}
 &\left|\Im\sum_{p\leq x^2}
 \frac{C_\pi(p)\bigl(\Lambda_x(p)p^{1/2-\sigma_x}-\Lambda(p)\bigr)}
 {p^{1/2+\i t}\log p}\right|^{2n}\\
 &=\frac{(-1)^n}{2^{2n}}\Biggl(
 \sum_{p\leq x^2}
 \frac{C_\pi(p)\bigl(\Lambda_x(p)p^{1/2-\sigma_x}-\Lambda(p)\bigr)}
 {p^{1/2+\i t}\log p}-\sum_{p\leq x^2}
 \frac{\overline{C_\pi(p)}\bigl(\Lambda_x(p)p^{1/2-\sigma_x}-\Lambda(p)\bigr)}
 {p^{1/2-\i t}\log p}\Biggr)^{2n},
\end{aligned}
\]
where $
 \Lambda_x(p)p^{1/2-\sigma_x}-\Lambda(p)
 =O((\log p)^2/\log x),
$
and, for $1\leq k\leq7$,
\[
\begin{aligned}
 &\left|\Im\sum_{p^k\leq x^2}
 \frac{C_\pi(p^k)\Lambda_x(p^k)}
 {p^{k(\sigma_x+\i t)}\log p^k}\right|^{2n}=\frac{(-1)^n}{2^{2n}}
 \left(\sum_{p^k\leq x^2}
\frac{C_\pi(p^k)\Lambda_x(p^k)}
 {p^{k(\sigma_x+\i t)}\log p^k}
 -\sum_{p^k\leq x^2}\frac{\overline{C_\pi(p^k)}\Lambda_x(p^k)}
 {p^{k(\sigma_x-\i t)}\log p^k}\right)^{2n},
\end{aligned}
\]
and
\[
\begin{aligned}
 \left|\sum_{p^k\leq x^2}
 \frac{C_\pi(p^k)\Lambda_x(p^k)}{p^{k(\sigma_x+\i t)}}\right|^{2n}&\ll_n
 \left|\Re\sum_{p^k\leq x^2}
 \frac{C_\pi(p^k)\Lambda_x(p^k)}{p^{k(\sigma_x+\i t)}}\right|^{2n}
 +\left|\Im\sum_{p^k\leq x^2}
 \frac{C_\pi(p^k)\Lambda_x(p^k)}{p^{k(\sigma_x+\i t)}}\right|^{2n}\\
 &\ll_n\left(\sum_{p^k\leq x^2}
\frac{C_\pi(p^k)\Lambda_x(p^k)}{p^{k(\sigma_x+\i t)}}
 +\sum_{p^k\leq x^2}\frac{\overline{C_\pi(p^k)}\Lambda_x(p^k)}{p^{k(\sigma_x-\i t)}}\right)^{2n}\\
 &+(-1)^n\left(\sum_{p^k\leq x^2}
\frac{C_\pi(p^k)\Lambda_x(p^k)}{p^{k(\sigma_x+\i t)}}
 -\sum_{p^k\leq x^2}\frac{\overline{C_\pi(p^k)}\Lambda_x(p^k)}{p^{k(\sigma_x-\i t)}}\right)^{2n}.
\end{aligned}
\]
Now
\[
 C_\pi(p^k)=\alpha_p^k+\beta_p^k+\gamma_p^k
\]
is a symmetric polynomial in the Satake parameters.
Since their elementary symmetric polynomials are
$A_\pi(1,p)$, $A_\pi(p,1)=\overline{A_\pi(1,p)}$, and $1$,
Newton's identities express $C_\pi(p^k)$ as a polynomial of total
degree at most $k$ in $A_\pi(1,p)$ and its conjugate.

We now apply the argument in Cases~(\uppercase\expandafter{\romannumeral 1})--(\uppercase\expandafter{\romannumeral 3})
of Section~\ref{the main term}.
In the third sum in \eqref{eq:remainder-decomposition}, we include the factor $1/\log x$ in each coefficient, so the coefficients of all three sums
are uniformly bounded. For fixed $k$ in the second and third sums,
expanding the $2n$-th powers gives at most $O_n(x^{4n/k})$ terms;
the first sum satisfies the same bound with $k=1$.
By \eqref{eq:Hecke-power-expansion} and
multiplicativity, each term reduces to Hecke pairs
$\overline{A_\pi(m_1,m_2)}A_\pi(n_1,n_2)$ with total absolute
coefficient $O_n(1)$ and
\[
 Q=m_1m_2n_1n_2\leq x^{4n}=N^{2n\delta}<N.
\]
Since $x^{4n/k}\leq x^{4n}$, Lemma~\ref{lem:new-orth} and
Corollary~\ref{cor:mass} give the total error 
\[
 O_{h,n,\varepsilon}\left(
 N^{-1+\varepsilon}x^{4n(1+\theta_2+\varepsilon)}+N^{-2}x^{8n}\right)
 =O_{h,n}(N^{-c})
\]
for some $c>0$ by our choice of $\delta$.
It remains to estimate the diagonal terms. For $k=1,2$, the identities
\[
 C_\pi(p)=A_\pi(1,p),\qquad
 C_\pi(p^2)=A_\pi(1,p^2)-A_\pi(p,1)
\]
show that a prime occurring only once gives no diagonal contribution.
The coefficients of the first sum and of the $k=1$ part
of the third sum are bounded by $(\log p)/(\sqrt p\log x)$,
and those of the $k=2$ parts of the second and third sums are bounded
by $1/p$. Their squared sums satisfy
\[
 \frac1{(\log x)^2}\sum_{p\leq x^2}\frac{(\log p)^2}{p}\ll1,
 \qquad\text{and}\quad
 \sum_{p\leq x}\frac1{p^2}\ll1.
\]
Since every prime in a diagonal term occurs at least twice,
these bounds give a total diagonal contribution $O_n(1)$.
For $3\leq k\leq7$, the coefficients in both sums are
bounded by $2p^{-k/2}/k$. Since $\sum_p p^{-k/2}<\infty$,
the sum of the absolute weights over all $2n$-tuples is
bounded. 
Together with the error estimate above, this shows that the three
spectral sums in \eqref{eq:remainder-decomposition}, including
their displayed factors, are $O_{h,n}(1)$.
Hence we obtain \eqref{eq:remainder-moment}.

\subsection{Proof of Theorem~\ref{main-theorem}}
By the binomial
theorem, we have
\[
\begin{aligned}
 \sum_{\pi\in\mathcal{B}(N)}\omega_{N,h}(\pi)S(t,\pi)^n
 &=\sum_{\pi\in\mathcal{B}(N)}\omega_{N,h}(\pi)\mathcal M(t,\pi)^n\\
 &\qquad+O_n\!\left(\sum_{\ell=1}^n\sum_{\pi\in\mathcal{B}(N)}
 \omega_{N,h}(\pi)|\mathcal M(t,\pi)|^{n-\ell}|\mathcal R(t,\pi)|^\ell\right),
\end{aligned}
\]
where $\mathcal M(t,\pi)$ and $\mathcal R(t,\pi)$ are defined in \eqref{MR def}.
Note that the argument in Section~\ref{the main term} also applies to the
$2n$-th moment for our choice of $\delta$.
Hence, for $1\leq\ell\leq n$, H\"older's inequality,
\eqref{eq:main-moments}, and \eqref{eq:remainder-moment} yield
\[
\begin{aligned}
 \sum_\pi\omega_{N,h}(\pi)|\mathcal M(t,\pi)|^{n-\ell}|\mathcal R(t,\pi)|^\ell
 &\leq
 \left(\sum_\pi\omega_{N,h}(\pi)|\mathcal M(t,\pi)|^{2n}\right)^{\frac{n-\ell}{2n}}
 \left(\sum_\pi\omega_{N,h}(\pi)|\mathcal R(t,\pi)|^{2n}\right)^{\frac{\ell}{2n}}\\
 &\ll_{h,t,n}(\log\log N)^{(n-\ell)/2}\ll_{h,t,n}(\log\log N)^{(n-1)/2}.
\end{aligned}
\]
This proves Theorem~\ref{main-theorem}.

\subsection{Proof of Corollary~\ref{cor:gaussian-limit}}
For each fixed $n\in\mathbb N$, the $n$-th moment of $\nu_{N,t,h}$ is
\[
 \int_{\mathbb R}u^n\,\mathrm{d}\nu_{N,t,h}(u)
 =\frac1{W_N}\sum_{\pi\in\cB(N)}\frac{h(\mu_\pi)}{L(1,\pi,\Ad)}
 \left(\frac{S(t,\pi)}{\sqrt{\log\log N}}\right)^n.
\]
By Remark~\ref{rem:A0}, the same $A_0$ and $h$ may be used for all
moment orders. Theorem~\ref{main-theorem} therefore gives
\[
 \lim_{\substack{N\to\infty\\N\ {\rm prime}}}
 \int_{\mathbb R}u^n\,\mathrm{d}\nu_{N,t,h}(u)
 =C_n=\sqrt{\pi}\int_{\mathbb R}u^n e^{-\pi^2u^2}\,\mathrm{d}u.
\]
Since the Gaussian measure $\nu$ is determined by its moments,
the method of moments \cite[Theorem~30.2]{Billingsley} yields the
claimed weak convergence. Corollary~\ref{cor:gaussian-limit} then follows from the
continuity of the Gaussian distribution function.

\section{Proof of Lemma~\ref{lem:new-orth}}\label{sec:proof-new-orth}
Applying Lemma~\ref{thm:prime-level-kuznetsov} to the pairs
$(m_1,m_2)$ and $(n_1,n_2)$ gives
\[
\frac{P_N}{N^2}\cC_{\new}=
 P_N H_h\,
 \delta_{n_1,m_1}\delta_{n_2,m_2}+\cS_4+\cS_5+\cS_6
 -\cC_{\old}-\cE_{\min}
 -\frac{P_N}{N^2}\cE_{\max,\new}-\cE_{\max,\old}.
\]
We will estimate each term on the right-hand side in turn.
The bounds below are uniform in $N$ and 
$m_1,m_2,n_1,n_2$, with implied constants depending at most on
$h$, $\eta$, and $\eps$.

\subsection{Short Weyl terms $\cS_4$ and $\cS_5$.}
\label{subsec:short-weyl}
Recall \eqref{gs}. The arithmetic conditions include
$ND_2\mid D_1$ and $m_2D_1=n_1D_2^2$.
Write
\[
 g=(m_2,n_1),\qquad m_2=ga,\qquad n_1=gb,
 \qquad (a,b)=1.
\]
By $ND_2\mid D_1$, write $D_1=NrD_2$ with $r\geq1$.
Then $arN=bD_2.$
Since $(a,b)=(b,N)=1$, we have $b\mid r$.  Writing $r=bd$ gives 
\[
 D_2=aNd,\qquad
 D_1=abN^2d^2,
\]
and therefore
\[
 D_1D_2=a^2bN^3d^3\geq N^3d^3.
\]
The argument $y$ of the $w_4$ transform is consequently bounded by
\[
 |y|=\frac{m_1m_2n_2}{D_1D_2}
 \leq\frac{Q}{N^3d^3}.
\]
Since $Q\leq N^\eta$ and $\eta<2$, using
\eqref{eq:Phi-w4-bound}, $P_N\asymp N^2$, and the trivial bound
\[
 |\widetilde S(-\epsilon n_2,m_2,m_1;D_2,D_1)|
\leq D_1D_2,
\]
we obtain
\[
|\cS_4|
 \mathrel\ll_{h,A}
 N^2\sum_{d\geq1}
 \left(\frac{Q}{N^3d^3}\right)^A
 \mathrel\ll_{h,A} Q^A N^{2-3A}.
\]
Taking $A>4/(3-\eta)$ gives $O_{h,\eta}(N^{-2})$.
For $\cS_5$, write $m_1=ga$, $n_2=gb$ with $(a,b)=1$. The
constraints give $D_1=aNd$ and $D_2=abN^2d^2$ for some $d\geq1$.
The same argument with the $w_5$ bound in
\eqref{eq:Phi-w4-bound} gives
$|\cS_5|\ll_{h,A}Q^A N^{2-3A}\ll_{h,\eta}N^{-2}$.

\subsection{Long Weyl term $\cS_6$.}
\label{subsec:long-weyl}
Let $y_1,y_2$ be the arguments of $\Phi_{w_6}$ in \eqref{gs}.
Then $|y_1y_2|=Q/(D_1D_2)$ and, since $D_1,D_2\geq1$,
$1+|y_1|+|y_2|\ll QD_1D_2$.  Moreover,
\[
 \min\{|y_1^2y_2|,|y_1y_2^2|\}
 \leq |y_1y_2|^{3/2}.
\]
For fixed $A>0$ and sufficiently large $A_0$,
\eqref{eq:Phi-w6-small-bound} gives
\begin{equation}\label{eq:w6-moduli-transform}
 |\Phi_{w_6}(y_1,y_2)|
 \ll_{h,A,\varepsilon}Q^{3A/2+\varepsilon}
 (D_1D_2)^{-3A/2+\varepsilon}.
\end{equation}
The factor $(QD_1D_2)^\varepsilon$ accounts for the loss
$(1+|y_1|+|y_2|)^\varepsilon$, also when the moduli are unbalanced.

For $X_1,X_2\geq N$, the averaged Kloosterman bound
\cite[Eq.\,(4.2)]{BBM}, applied to the arguments in \eqref{gs}, states that
\begin{equation}\label{average bound}
\begin{split}
 \sum_{\substack{N\mid D_1,\ D_1\leq X_1\\
                 N\mid D_2,\ D_2\leq X_2}}
 |S^{(N)}(\epsilon_2n_2,\epsilon_1n_1,m_1,m_2;D_1,D_2)|
 \ll_\varepsilon\frac{(X_1X_2)^{3/2+\varepsilon}Q^\varepsilon}{N^{3/2}}.
\end{split}
\end{equation}
Its coprimality hypothesis is exactly $(Q,N)=1$.
Split $D_i$ into dyadic ranges $D_i\asymp X_i=2^{j_i}N$.
Combining \eqref{eq:w6-moduli-transform}, 
\eqref{average bound} with $P_N\asymp N^2$, we obtain
\[
 |\cS_6(X_1,X_2)|
 \ll_{h,A,\varepsilon}N^{1/2}Q^{3A/2+\varepsilon}
 (X_1X_2)^{(1-3A)/2+\varepsilon}.
\]
Summing the dyadic blocks for sufficiently large $A$ gives
\[
 |\cS_6|\ll_{h,A,\varepsilon}
 N^{3/2-3A+\varepsilon}Q^{3A/2+\varepsilon}
 \ll_{h,A,\varepsilon,\eta}
 N^{3/2-3A(1-\eta/2)+\varepsilon}.
\]
As $\eta<2$, a sufficiently large fixed $A$ and a sufficiently
small $\varepsilon>0$ make this $O_{h,\eta}(N^{-2})$.
Choose $A_0$ sufficiently large, depending only on
$\eta$, so that the transform bounds used for
$\cS_4,\cS_5$, and $\cS_6$ hold with the fixed exponents above.

\subsection{Cuspidal oldforms.}
For a level $1$ representation $\Pi$, \cite[Proposition~5]{BC} gives the three unorthogonalized
translates at an index $(m,n)$ as
\begin{align*}
 A_\Pi^{(0)}(m,n)&=A_\Pi(m,n),\\
 A_\Pi^{(1)}(m,n)&=N^{1/2}
 \left(A_\Pi(m/N,Nn)+A_\Pi(m,n/N)\right),\\
 A_\Pi^{(2)}(m,n)&=N A_\Pi(m/N,n).
\end{align*}
A coefficient with a non-integral argument is zero. Since each of
$n_1,n_2,m_1,m_2$ is coprime to $N$, both translated coefficients
$A_\Pi^{(1)}$ and $A_\Pi^{(2)}$ vanish at the two indices used here.
\cite[Corollary~4]{BC} gives the coefficients $B_\Pi^{(j)}$ as follows:
\begin{align*}
 B_\Pi^{(0)}(m,n)&=A_\Pi^{(0)}(m,n),\\
 B_\Pi^{(1)}(m,n)&=(1+c_{11})A_\Pi^{(1)}(m,n)+c_{10}A_\Pi^{(0)}(m,n),\\
 B_\Pi^{(2)}(m,n)&=(1+c_{22})A_\Pi^{(2)}(m,n)
       +c_{21}A_\Pi^{(1)}(m,n)+c_{20}A_\Pi^{(0)}(m,n).
\end{align*}
Consequently, at either of our indices,
\begin{equation}\label{eq:old-coprime}
 B_\Pi^{(0)}(m,n)=A_\Pi(m,n),\qquad
 B_\Pi^{(1)}(m,n)=c_{10}A_\Pi(m,n),\qquad
 B_\Pi^{(2)}(m,n)=c_{20}A_\Pi(m,n).
\end{equation}
For cuspidal $\Pi$, the
Gram--Schmidt bounds are $|c_{10}|+|c_{20}|\ll N^{\theta_3-1/2}$,
where $\theta_3=5/14<1/2$.  Thus
\[
 \begin{aligned}
\sum_{j=0}^{2}
 \left|\overline{B_\Pi^{(j)}(m_1,m_2)}
 B_\Pi^{(j)}(n_1,n_2)\right|
&\ll (1+N^{2\theta_3-1})
 |A_\Pi(m_1,m_2)A_\Pi(n_1,n_2)|\\
& \ll |A_\Pi(m_1,m_2)A_\Pi(n_1,n_2)|.
 \end{aligned}
\]
For every prime $p$, by \eqref{eq:local-Schur},
$A_\Pi(p^r,p^u)$ is a Schur polynomial of degree $2r+u$
with $\tfrac12(r+1)(u+1)(r+u+2)$ monomials, counted with multiplicity.
The Kim--Sarnak bound
$|\alpha_{\Pi,j}(p)|\leq p^{\theta_3}$
and multiplicativity thus give the uniform estimate
$|A_\Pi(m,n)|\ll_\eps (mn)^{2\theta_3+\eps}$.

Put $T_\Pi=1+\|\mu_\Pi\|$.
By Brumley \cite[Theorem~3]{Brumley} and Miller
\cite[Corollary~1.4]{Miller}, respectively, for some $B>0$ and
$T\geq1$ we have
\[
 L(1,\Pi,\Ad)^{-1}\ll T_\Pi^B,
 \qquad
 \#\{\Pi\in\mathcal B(1):T_\Pi\leq T\}\ll T^5.
\]
Here
$\operatorname*{Res}_{s=1}L(s,\Pi\times\widetilde\Pi)
=L(1,\Pi,\Ad)$, and the analytic conductor of
$\Pi\times\widetilde\Pi$ at level $1$ is bounded by a fixed
power of $T_\Pi$, giving the first estimate.
For the second, the Laplace eigenvalue satisfies
$
 \lambda_\Pi=1-\frac12\sum_{j=1}^3\mu_{\Pi,j}^2
 \ll T_\Pi^2,
$
so Miller's eigenvalue counting bound yields the stated
bound in terms of $T_\Pi$.
Since $|h(\mu_\Pi)|\ll_{h,D}T_\Pi^{-D}$ for every $D>0$,
dyadic summation with $D>B+5$ shows that
\[
 \sum_{\Pi\in\mathcal B(1)}
 \frac{|h(\mu_\Pi)|}{L(1,\Pi,\Ad)}
 \ll_{h,D}\sum_{k\geq0}2^{k(B+5-D)}<\infty.
\]
Consequently,
\begin{equation}\label{eq:cusp-old-bound}
 |\cC_{\old}|\ll_{h,\eps} Q^{2\theta_3+\eps}.
\end{equation}

\subsection{Minimal Eisenstein spectrum.}
By \cite[Proposition~6]{BC}, the minimal Eisenstein oldvectors
satisfy \eqref{eq:old-coprime} at our indices, with
$|c_{10}|+|c_{20}|\ll N^{-1/2}$.  Their unitary Satake
parameters and the Schur formula give
$|A_s(m,n)|\ll_\eps(mn)^\eps$ uniformly on the tempered plane.
Hence
\eqref{eq:old-coprime} gives, for the pair of indices in
\eqref{eq:Emin},
\[
 \sum_{j=0}^{2}
 \bigg|\overline{B_s^{(j)}(m_1,m_2)}B_s^{(j)}(n_1,n_2)\bigg|
 \ll (1+N^{-1})
 |A_s(n_1,n_2)A_s(m_1,m_2)|\ll_\eps Q^\eps.
\]
Write $s_i=\i t_i$, where $t_3=-t_1-t_2$.
The classical bound
$|\zeta(1+\i t)|^{-1}\ll_\varepsilon(1+|t|)^\varepsilon$ for real $t$ 
(see \cite[Eq.\,(3.11.10)]{Titchmarsh}),
with $1/\zeta(1)=0$ understood by continuity,
gives that the three squared reciprocal factors in
\eqref{eq:Emin} satisfy
\[
 \prod_{1\leq i<j\leq3}
 |\zeta(1+\i(t_i-t_j))|^{-2}
 \ll_\varepsilon (1+|t_1|+|t_2|)^{6\varepsilon}.
\]
For every $A>0$, the defining decay of $h$ gives
$|h(\i t_1,\i t_2,\i t_3)|\ll_{h,A}(1+|t_1|+|t_2|)^{-A}$.
Thus, taking $A>2+6\varepsilon$ in \eqref{eq:Emin} yields
\begin{equation}\label{eq:minimal}
 |\cE_{\min}|\ll_{h,\eps}Q^\eps
 \int_{\mathbb R^2}(1+|t_1|+|t_2|)^{-A+6\varepsilon}\,\mathrm{d}t_1\,\mathrm{d}t_2
 \ll_{h,\eps}Q^\eps.
\end{equation}

\subsection{Maximal Eisenstein old spectrum.}
For a level $1$ inducing $\GL_2$ Maass form $\phi$ and
$s=\i\tau$, \cite[Proposition~7(b)]{BC} gives the oldvector reduction
\eqref{eq:old-coprime} with
$|c_{10}|+|c_{20}|\ll N^{\theta_2-1/2}$, uniformly in $\phi$
and $\tau$. As before, the two nontrivial translates have zero
coefficient at every index coprime to $N$. 
In each monomial of the Schur polynomial
\eqref{eq:local-Schur}, the exponents of the two $\GL_2$
parameters differ by at most $r+u$.  Their moduli are reciprocal,
so the Kim--Sarnak bound \cite{KS} gives, at every prime $p$
and uniformly in $\phi$ and $\tau$,
\[
 |A_{(\phi,\i\tau)}(p^r,p^u)|
 \leq \frac{(r+1)(u+1)(r+u+2)}{2}\,p^{\theta_2(r+u)}.
\]
Multiplicativity then yields
$|A_{(\phi,\i\tau)}(m,n)|\ll_\eps(mn)^{\theta_2+\eps}$.
Since
$2\theta_2-1<0$, \eqref{eq:old-coprime} implies
\[
 \begin{aligned}
\sum_{j=0}^{2}
 \bigg|\overline{B_{(\phi,\i\tau)}^{(j)}(m_1,m_2)}B_{(\phi,\i\tau)}^{(j)}(n_1,n_2)\bigg|
\ll (1+N^{2\theta_2-1})
 |A_{(\phi,\i\tau)}(n_1,n_2)
   A_{(\phi,\i\tau)}(m_1,m_2)|\ll_\eps Q^{\theta_2+\eps}.
 \end{aligned}
\]
Moreover, for every $\varepsilon>0$, we have
\bna\label{estimate}
 L(1,\phi,\Ad)^{-1}\ll_\varepsilon(1+|t_\phi|)^\varepsilon,
 \qquad
 |L(1+3\i\tau,\phi)|^{-2}
 \ll_\varepsilon(1+|t_\phi|+|\tau|)^\varepsilon,
\ena
by \cite[Theorem~0.2]{HoffsteinLockhart}
and \cite[Theorem~C(2)]{HoffsteinRamakrishnan},
respectively, applying the latter to
$\phi\otimes|\det|^{3\i\tau}$.
The adjoint $L$-value bound is used without exceptions, with a
possibly ineffective implied constant.
For nonvanishing and related lower bounds, see also
\cite{JacquetShalikaNV,GLS}.
Since
$\mu_{\phi,\i\tau}=(\i\tau+t_\phi,\i\tau-t_\phi,-2\i\tau)$,
we have
$\|\mu_{\phi,\i\tau}\|^2=2|t_\phi|^2+6\tau^2$, and hence
$1+\|\mu_{\phi,\i\tau}\|\asymp1+|t_\phi|+|\tau|$.
Thus $|h(\mu_{\phi,\i\tau})|\ll_{h,A}(1+|t_\phi|+|\tau|)^{-A}$ for every $A>0$,
also for exceptional $t_\phi$, since $|\Re(t_\phi)|\leq7/64$.
The level $1$ Weyl bound \cite[Corollary~3.2.3, pg.\,47]{Palm} gives
$\#\{\phi:|t_\phi|\leq T\}\ll(1+T)^2$.
Consequently, taking absolute values in \eqref{eq:Emax-old}
and summing over $1+|t_\phi|\asymp2^k$, we obtain
\[
\begin{aligned}
 |\cE_{\max,\old}|
 &\ll_{h,A,\eps}Q^{\theta_2+\eps}
 \sum_{\phi}\int_{\mathbb R}(1+|t_\phi|+|\tau|)^{-A+2\varepsilon}\,\mathrm{d}\tau\\
 &\ll_{h,A,\eps}Q^{\theta_2+\eps}
 \sum_{k\geq0}2^{k(3+2\varepsilon-A)}.
\end{aligned}
\]
Choosing $A>3+2\varepsilon$ proves
\begin{equation}\label{eq:max-old-bound}
 |\cE_{\max,\old}|\ll_{h,\eps}Q^{\theta_2+\eps}.
\end{equation}

\subsection{Maximal Eisenstein new spectrum.}
Let $\phi$ be a level $N$ $\GL(2)$ newform and $s=\i\tau$.
Since $(Q,N)=1$, all primes in the coefficient indices are
unramified.  The Schur-weight argument above gives
\bea\label{the coefficient bound}
 |A_{(\phi,s)}(n_1,n_2)A_{(\phi,s)}(m_1,m_2)|
 \ll_\eps Q^{\theta_2+\eps}.
\eea
By \cite[Theorem~0.2]{HoffsteinLockhart}
and
\cite[Theorem~C(2)]{HoffsteinRamakrishnan}, for every $\eps>0$,
\bea\label{new estimate}
 L(1,\phi,\Ad)^{-1}\ll_\eps N^\eps(1+|t_\phi|)^\eps,\qquad
 |L(1+3\i\tau,\phi)|^{-2}\ll_\eps
 N^\eps(1+|t_\phi|+|\tau|)^\eps.
\eea
Palm's uniform Weyl law
\cite[Corollary~3.2.3, pg.\,47]{Palm}, applied to
$\Gamma_0^{\pm}(N)=\{\gamma\in\GL_2(\mathbb Z):N\mid\gamma_{21}\}$
of index $N+1$ in $\GL_2(\mathbb Z)$, gives
\[
 \#\{\phi\text{ newforms of level }N:|t_\phi|\leq T\}
 \ll N(1+T)^2
\]
uniformly for primes $N$ and $T\geq0$, including exceptional
parameters.
Since $1+\|\mu_{\phi,\i\tau}\|\asymp1+|t_\phi|+|\tau|$,
\eqref{the coefficient bound}, \eqref{new estimate}, and the rapid
decay of $h$ bound the absolute value of the integrand in
\eqref{eq:Emax-new}, for any fixed $A>3$, by
\[
 \ll_{h,A,\eps}N^\eps Q^{\theta_2+\eps}
 (1+|t_\phi|+|\tau|)^{-A}.
\]
Integrating in $\tau$ and applying the Weyl bound on each dyadic
range $1+|t_\phi|\asymp2^k$, we obtain
\begin{equation}\label{max new bound}
\begin{aligned}
 |\cE_{\max,\new}|
 &\ll_{h,A,\eps}N^{1+\eps}Q^{\theta_2+\eps}
 \sum_{k\geq0}2^{k(3-A)}
 \ll_{h,\eps}N^{1+\eps}Q^{\theta_2+\eps},
\end{aligned}
\end{equation}
where $A>3$ is fixed.

\subsection{Conclusion.}
Sections~\ref{subsec:short-weyl}--\ref{subsec:long-weyl} give
$\cS_4+\cS_5+\cS_6=O_{h,\eta}(N^{-2})$.
The bounds \eqref{eq:cusp-old-bound}, \eqref{eq:minimal}, and
\eqref{eq:max-old-bound} give a total contribution
$O_{h,\varepsilon}(Q^{2\theta_3+\varepsilon})$, since $\theta_2<2\theta_3$.
Together with \eqref{max new bound} and $Q\geq1$,
\eqref{eq:full-kuznetsov}, after division by $P_N/N^2$, yields
\[
 \cC_{\new}=N^2H_h\delta_{n_1,m_1}\delta_{n_2,m_2}
 +O_{h,\eta,\eps}\bigl(N^{1+\eps}Q^{\theta_2+\eps}+Q^{2\theta_3+\varepsilon}\bigr),
\]
Since $P_N-N^2=N+1$, replacing $N^2H_h$ by $P_NH_h$
changes the diagonal term by $O_h(N)$, which is absorbed by the
error term above, giving \eqref{eq:neworth}.

\bibliographystyle{amsplain}

\end{document}